\documentclass[12pt]{amsart}

\usepackage{etex}
\usepackage{graphicx,amsthm}
\usepackage{enumerate}
\usepackage{empheq}
\usepackage[margin=1in]{geometry}
\usepackage{pst-all}
\usepackage{array}
\usepackage{tikz}
\usepackage{tikz-cd}
\usepackage{amsmath,amscd,amsfonts,epic,eepic,bbm, mathabx,graphicx,ytableau}
\allowdisplaybreaks[4]
\usepackage[onehalfspacing]{setspace}
\usepackage{float}

\usepackage[all]{xy} 

\usepackage{subfigure}

\usepackage{amssymb, wasysym}
\usepackage{ccaption}

\usepackage{setspace}

\usepackage{caption}
\usepackage{diagbox}
\usepackage{url}

\usepackage[mathscr]{euscript} 
\usepackage{booktabs} 

\allowdisplaybreaks

\newtheorem{theorem}{Theorem}[section]
\newtheorem{lemma}[theorem]{Lemma}
\newtheorem{proposition}[theorem]{Proposition}

\newtheorem{conjecture}[theorem]{Conjecture}

\newtheorem{definition}[theorem]{Definition}
\newtheorem{example}[theorem]{Example}

\newtheorem{remark}[theorem]{Remark}

\numberwithin{equation}{section}

\newcommand{\la}{\lambda}

\newcommand{\x}{\mathbf{x}}
\renewcommand{\a}{\alpha}
\renewcommand{\b}{\beta}
\renewcommand{\c}{\gamma}
\renewcommand{\d}{\delta}
\newcommand{\e}{\epsilon}

\newlength\cellsize 

\savebox2{%
\begin{picture}(18,18)\linethickness{0.6pt} %
\put(0,0){\line(1,0){18}}
\put(0,0){\line(0,1){18}}
\put(18,0){\line(0,1){18}}
\put(0,18){\line(1,0){18}}
\end{picture}}

\newcommand\cellify[1]{\def\thearg{#1}\def\nothing{}%
\ifx\thearg\nothing\vrule width0pt height\cellsize depth0pt%
  \else\hbox to 0pt{\usebox2\hss}\fi%
  \vbox to \cellsize{\vss\hbox to \cellsize{\hss$_{#1}$\hss}\vss}}

\newcommand\tableau[1]{\vtop{\let\\=\cr
\setlength\baselineskip{-1000pt}
\setlength\lineskiplimit{1000pt}
\setlength\lineskip{0pt}
\ialign{&\cellify{##}\cr#1\crcr}}}

\newlength{\celldim} 
\newsavebox{\cell}

\sbox{\cell}{%
\begin{picture}(22,22)\linethickness{0.6pt} %
  \put(0,0){\line(1,0){22}} \put(0,0){\line(0,1){22}}
  \put(22,0){\line(0,1){22}} \put(0,22){\line(1,0){22}}
\end{picture}}

\newcommand\cellifying[1]{%
  \def\thearg{#1}\def\nothing{}%
  \ifx\thearg\nothing \vrule width0pt height\celldim depth0pt\else
  \hbox to 0pt{\usebox{\cell} \hss}\fi%
  \vbox to \celldim{ \vss \hbox to
  \celldim{\hss$#1$\hss} \vss}
}
\newcommand\ttableau[1]{\vtop{\let\\\cr
\baselineskip -16000pt \lineskiplimit 16000pt \lineskip 0pt
\ialign{&\cellifying{##}\cr#1\crcr}}}

\newcommand\bas[1]{\omit \vbox to \cellsize{ \vss \hbox to \celldim{\hss$#1$\hss} \vss}}

\begin{document}

\title[Positivity of chromatic symmetric functions]{Positivity of chromatic symmetric functions associated with Hessenberg functions of bounce number $3$}
\thanks{The work on this paper was done while the first named author was visiting Korea Institute for Advanced Study(KIAS). She is grateful to KIAS for the hospitality.}

\author{Soojin Cho}
\address{Department of Mathematics, Ajou University, Suwon  16499, Republic of Korea}
\email{chosj@ajou.ac.kr}

\author{Jaehyun Hong}
\address{Center for Complex Geometry, Institute for Basic Science (IBS), Daejeon 34126, Republic of Korea}
\email{jhhong00@ibs.re.kr}

\begin{abstract}
We give a proof of the Stanley-Stembridge conjecture on chromatic symmetric functions for the  class of all unit interval graphs with independence number $3$. That is, we show that the chromatic symmetric function of the incomparability graph of a unit interval order in which the length of a chain is at most $3$
is positively expanded as a linear sum of elementary symmetric functions.
\end{abstract}



\maketitle
\section{Introduction} \label{sec:intro}

A \emph{chromatic symmetric function} $X_G(\mathbf{x})$ of a finite simple graph $G$ with vertex set $V$ is  defined in a natural way to generalize the chromatic polynomials;
$$X_G(\mathbf{x})=\sum_\kappa \left( \prod_{v\in V} x_{\kappa(v)}\right),$$
where the sum is over all proper colorings $\kappa: V \rightarrow \mathbb P$ of $G$ with the set of positive integers $\mathbb P$.
Since chromatic symmetric functions  were introduced by Stanley in 1995 \cite{S1}, they have become an important area of research in the relations to many different fields including combinatorics, representation theory and algebraic geometry.

The Stanley-Stembridge conjecture \cite{SS, S1} is a well known open conjecture on chromatic symmetric functions which states that chromatic symmetric functions of the incomparability graph of a (3+1)-free poset can be positively expanded as a sum of elementary symmetric functions, i.e. is \emph{$e$-positive}.

Gasharov \cite{Gasha} proved that the chromatic symmetric functions in the Stanley-Stembridge conjecture can be positively expanded as a sum of Schur functions by constructing combinatorial objects called $P$-tableaux, that is a weaker result than the conjecture since elementary symmetric functions are positively expanded as a sum of Schur functions.
 since elementary symmetric functions are positively expanded as a sum of Schur functions. The conjecture itself was proved for some special classes of graphs including the complement of a bipartite graph that was considered in \cite{S1} and recently in \cite{CH, HP}; see Remark 2.18 in \cite{CH} for a list of graphs with which the conjecture has been proved to be true. An important result concerning the Stanley-Stembridge conjecture due to Guay-Paquet \cite{G-P1} is that it is enough to prove the conjecture for all posets that are both $(3+1)$-free \emph{and} $(2+2)$-free. This reduces the class of objects we have to consider for the proof of the conjecture down to the posets of \emph{unit interval orders}.

The maximum length of possible chains in a unit interval order plays an important role in understanding the $e$-expansion of the corresponding chromatic symmetric function, and we prove the Stanley-Stembridge conjecture for the unit interval orders in which  the maximum length of chains is at most $3$.
We note that the chains in the corresponding unit interval orders of the complements of bipartite graphs, in which the conjecture was proved to be true as stated above, have  length at most $2$. We also note that the Stanley-Stembridge conjecture was proved only for a few special cases when the longest chain in the unit interval poset has length $3$. (See Section 3.2 of \cite{CH}.)  We follow and generalize the basic idea in \cite{CH} to use Gasharov's $P$-tableaux for the Schur expansion of the chromatic symmetric functions and Jacobi-Trudi identity for the proof of the main theorem of the current paper. We write the coefficients in the $e$-expansion of the chromatic symmetric functions as a sum of signed sets of $P$-tableaux of possible shapes that correspond to permutations in the symmetric group $\mathfrak{S}_3$. Then we construct injective maps from negative sets to positive sets to complete the proof. Our work to write the coefficients as a sum of signed sets can be extended to the general case, while the construction of injective maps needs more fine work with insight.

Shareshian and Wachs \cite{SW} defined a quasisymmetric refinement of chromatic symmetric functions and introduced the \emph{natural unit interval orders} as (naturally labeled) representatives of the classes of equivalent unit interval orders. Then, in terms of natural unit interval orders they derived a refined Gasharov's result, generalized the Stanley-Stembridge conjecture to its quasi form,
and made a conjecture that their chromatic quasisymmetric functions, after we apply the usual involution $\omega$ on symmetric functions,
are the Frobenius characteristics of the symmetric group representations derived from the Tymoczko's dot action \cite{T1, T2} on the cohomology of Hessenberg varieties of type $A$.

It is remarkable that the Shareshian-Wachs conjecture was proved to be true independently by Brosnan and Chow \cite{BC}, and Guay-Paquet \cite{G-P2}. 
This enables one to understand the $e$-positivity conjecture by Stanley-Stembridge as the $h$-positivity, where $h$ stands for the \emph{homogeneous} symmetric functions, of the symmetric group representation
on the cohomology of Hessenberg varieties. We also have to note that the notion of \emph{natural} unit interval orders is closely related with Hessenberg varieties through the Hessenberg functions or equivalently the Dyck paths.

With all of these profound theories developed so far on the chromatic symmetric functions, especially on the conjecture by Stanley-Stembridge, we could describe the conjecture in terms of \emph{Hessenberg functions}. 
In the rest of this section, starting with a definition of Hessenberg functions we proceed to state the Stanley-Stembridge conjecture in Conjecture~\ref{conj:e-positive} and \ref{conj:h-positive} and finally give a statement of our main theorem, Theorem~\ref{thm:main}. We adopt the $h$-positivity statement of the conjecture for our argument since that makes it easier to handle Gasharov's $P$-tableaux for the proof of the main theorem.

\begin{definition} For a positive integer $n$, a non-decreasing function $f \,:\, [n] \rightarrow [n]$ is called a \emph{Hessenberg function} if $i \leq f(i)$ for all $i\in [n]$ where $[n]$ is the set $\{1, 2, \dots, n\}$.
\end{definition}

\begin{definition}\label{def:hess_related} Let $f : [n] \rightarrow [n]$ be a Hessenberg function for a positive integer $n$.
\begin{itemize}
\item The \emph{natural unit interval order} $P(f)$ associated with $f$ is the poset on $[n]$ with the order relation $\prec_f$ defined by
 $$i\prec_f j \mbox{ if and only if } f(i)<j\,.$$

\item The \emph{natural unit interval graph} $G(f)$ associated with $f$ is the graph on the vertex set $[n]$ where
$$\{i, j\}, i<j, \mbox{ is an edge of } G(f) \mbox{ if and only if } i\nprec_f j \mbox{ or equivalently } f(i)\geq j\,.$$

\item The \emph{Hessenberg variety} $\mathcal{H}(f, s)$ associated with $f$ and a linear transformation $s : \mathbb C^n \rightarrow \mathbb C^n$ is the set of complete flags defined as follows;
$$\mathcal{H}(f, s)=\{F_0\subset F_1 \subset \cdots \subset F_n=\mathbb C^n\,|\, \mathrm{dim}F_i=i, \, sF_i\subseteq F_{f(i)}\,\mbox{ for all } i\in [n] \}\,.$$
\end{itemize}
\end{definition}

It is well known that unit interval orders are characterized as $(3+1)$-free and $(2+2)$-free posets and the number of isomorphism classes of unit interval orders is the Catalan number. The poset $P(f)$ in Definition~\ref{def:hess_related} are \emph{naturally} labeled unit interval orders, which are representatives of the isomorphism classes of unit interval orders. (See Section 4 of \cite{SW} and the references therein for a detailed explanation on unit interval orders.)  We also note that the natural interval graph $G(f)$ is the \emph{incomparability graph} of the natural unit interval order $P(f)$, and therefore the independence number (the maximum size of an induced subgraph that has no edge), of $G(f)$ coincides with the length of the longest chain of $P(f)$.

 There are many equivalent descriptions to define natural unit interval orders and the following proposition is from one of them.

\begin{proposition}[Proposition 4.1 in \cite{SW}] \label{prop:unit_interval_order} Let  $f : [n] \rightarrow [n]$ be a Hessenberg function.
\begin{enumerate}
\item If $i \prec_f j$ then $i<j$ in the natural order on the integers.
\item If the direct sum(disjoint union) $\{i \prec_f k\}+\{j\}$ is an induced subposet of $P(f)$ then $i<j<k$ in the natural order on the integers.
\end{enumerate}
\end{proposition}

We fix a set of infinitely many variables $\x=(x_1, x_2, \dots)$ and consider the algebra of symmetric functions $\Lambda(\mathbf x)$ over a field. For a given positive integer $k$, the $k$th \emph{elementary symmetric function} $e_k$ and the $k$th \emph{homogeneous symmetric function} $h_k$ are defined as $$e_k=\sum_{i_1<\cdots<i_k} x_{i_1}\cdots x_{i_k}\quad \mbox{ and }\quad h_k=\sum_{i_1\leq\cdots\leq i_k} x_{i_1}\cdots x_{i_k}\,.$$
A non-increasing sequence of positive integers $\la=(\la_1,\dots, \la_\ell)$ is  a \emph{partition} of $n=\sum_i \la_i$ whose \emph{length} $\ell(\la)$ is $\ell$, and we use $\la \vdash n$ to denote that $\la$ is a partition of $n$.  We use $\la'=(\la'_1, \dots, \la'_{\ell'})$ for the \emph{conjugate} of  $\la=(\la_1,\dots, \la_\ell)\vdash n$, that is $\la'_i=| \{ j\,|\, \la_j \geq i \} |$ for each $i$.
For a partition $\la=(\la_1,\dots, \la_\ell)$,  we let
\begin{equation} e_\la=e_{\la_1}\cdots e_{\la_\ell}\,, \mbox{ the elementary symmetric function,}\end{equation}
\begin{equation} h_\la=h_{\la_1}\cdots h_{\la_\ell}\,,\mbox{ the homogeneous symmetric function,}\end{equation}
\begin{align} \label{eq:Jacobi-Trudi}
s_\la &=\mathrm{det}(e_{\la'_i-i+j})_{\ell'\times\ell'}=\mathrm{det}(h_{\la_i-i+j})_{\ell\times\ell}\,,  \mbox{ the Schur function, }\\
&\mbox{ where  $e$ and $h$ with negative subscripts are $0$ and $e_0=h_0=1$.}\nonumber
\end{align}
Then
$\{ e_\mu \,|\, \mu\vdash n\}$,  $\{ h_\mu \,|\, \mu\vdash n\}$ and $\{ s_\mu \,|\, \mu\vdash n\}$ are
well known bases of the space $\Lambda^n(\mathbf x)$  of symmetric functions of degree $n$. The above definitions of Schur functions are known as \emph{Jacobi-Trudi identities}.
The algebra involution $\omega$ is defined by $\omega(e_i)=h_i$  or equivalently by $\omega(s_\la)=s_{\la'}$ on $\Lambda(\mathbf x)$, which explains the equivalence of two kinds of Jacobi-Trudi identities for $s_\la$. The \emph{Frobenius characteristic} $\mathrm{ch}$ is the map from the space of representations of $\mathfrak{S}_n$ to the space $\Lambda^n (\mathbf x)$ of symmetric functions sending the natural permutation representation corresponding to a partition $\lambda$ to the  homogeneous symmetric function $h_\lambda$.
For more details on symmetric functions and the representation of the symmetric group, the readers are referred to \cite{Mac}.

We let $\mathbb P$ be the set of positive integers. For a graph $G=(V,E)$ with vertex set $V$ and edge set $E$, a \emph{proper coloring of $G$} is a map $\kappa: V \rightarrow \mathbb P$ such that $\kappa(i)\ne \kappa(j)$ whenever $\{i, j\}\in E$. The chromatic symmetric function of a graph $G$ was defined by Stanley in \cite{S1} and its refinement by Shareshian-Wachs in \cite{SW}.

\begin{definition}[\cite{S1, SW}] \label{def:chromatic_quasi} When $G=([n], E)$ is a graph on the vertex set $[n]$, the \emph{chromatic quasisymmetric function of $G$} is
$$X_G(\mathbf{x}, t)=\sum_\kappa t^{asc(\kappa)} \,\mathbf x_\kappa\,, $$
where the sum is over all proper colorings of $G$, $\mathrm{asc}(\kappa)=|\{\{i, j\}\in E\,|\, i<j \mbox{ and } \kappa(i)<\kappa(j)\}|$ and $\mathbf x_\kappa=\prod_{i=1}^n x_{\kappa(i)}$.
\end{definition}

We remark that $X_G(\mathbf{x}, 1)$ is Stanley's chromatic symmetric function $X_{G}(\mathbf{x})$. Moreover,  $X_G(\mathbf{x}, t)$ is not a symmetric function in general while Shareshian and Wachs showed that  the chromatic quasisymmetric function  $X_{G(f)}(\mathbf{x}, t)$ of a natural unit interval graph $G(f)$  is a symmetric function in $\x=(x_1, x_2, \dots)$. Shareshian and Wachs conjectured the following in \cite{SW}, which has been proved by Brosnan-Chow \cite{BC} and Guay-Paquet \cite{G-P2} independently.

\begin{proposition} Let $f : [n] \rightarrow [n]$ be a Hessenberg function and $s: \mathbb C^n \rightarrow \mathbb C^n$ be a linear transformation with $n$ distinct eigenvalues, then
\[ \sum_j \mathrm{ch} H^{2j}(\mathcal H(f, s))\,t^j=\omega X_{G(f)}(\mathbf x, t)\,.\]
\end{proposition}

A long standing conjecture on chromatic (quasi)symmetric function is about \emph{positivity}, whose proof is known only for some special cases; see \cite{CH, DW,  GS, HP}.

\begin{conjecture}[\cite{S1, SW}]\label{conj:e-positive} For a given Hessenberg function $f : [n] \rightarrow [n]$, $X_{G(f)}(\mathbf{x}, t)$ is $e$-positive. That is, if we write $X_{G(f)}(\mathbf{x}, t)=\sum_\la b_\la(t) e_\la(\x)$, then $b_\la (t)$ is a polynomial of nonnegative integer coefficients.
\end{conjecture}

Conjecture~\ref{conj:e-positive} on $e$-positivity of $X_{G(f)}(\x, t)$ is equivalent to $h$-positivity of $\omega X_{G(f)}(\mathbf{x}, t)$.

\begin{conjecture}[\cite{S1, SW}]\label{conj:h-positive} For a given Hessenberg function $f : [n] \rightarrow [n]$, $\omega X_{G(f)}(\mathbf{x}, t)$ is $h$-positive. That is, if we write $\omega X_{G(f)}(\mathbf{x}, t)=\sum_\la c_\la(t) h_\la(\x)$, then $c_\la (t)$ is a polynomial of nonnegative integer coefficients.
\end{conjecture}

Our main work in the present paper is to give a proof of Conjecture~\ref{conj:h-positive} when $t=1$ and $P(f)$ does not have a chain with $4$ elements.

\begin{theorem}\label{thm:main}
For a given Hessenberg function $f : [n] \rightarrow [n]$ such that a longest chain in the poset $P(f)$ has $3$ elements, $\omega X_{G(f)}(\mathbf{x})$ is $h$-positive. That is, if we write $\omega X_{G(f)}(\mathbf{x})=\sum_\la c_\la h_\la(\x)$, then $c_\la$ is a nonnegative integer.
\end{theorem}

The rest of this paper is organized as follows.
We introduce Gasharov's result on the Schur-expansion of chromatic symmetric functions of $(3+1)$-free posets as well as important properties of the corresponding natural unit interval orders in Section~\ref{sec:Prelim}. In Section~\ref{sec:bounce3} we work on the $h$-expansion of the dual chromatic symmetric functions to recognize each coefficient as a signed sum of numbers of dual $P$-tableaux of certain types. In Section~\ref{sec:h-positivity} we prove the main theorem by constructing sign reversing injections from the negative sets to positive ones, while  some technical proofs are done in Section~\ref{sec:proofs}.

\section{Preliminaries}\label{sec:Prelim}

There is important work concerning Conjecture~\ref{conj:e-positive} by Gasharov, which shows that $X_{G(f)}(\mathbf{x})$ is expanded with nonnegative coefficients in terms of Schur functions. Note that this is a consequence of  Conjecture~\ref{conj:e-positive} since $e_\la$ is Schur positive. We state Gasharov's result in its dual form by taking the conjugates, so the following definition is the dual notion of  \emph{$P$-tableau} of Gasharov \cite{Gasha}.

\begin{definition}\label{def:f-tableau}
For a Hessenberg function $f : [n]\rightarrow [n]$ and $\la\vdash n$, an \emph{$f$-tableau of shape  $\la$} is a filling of the diagram of $\la$ with $1, 2, \dots, n$ such that

i) each column is strictly increasing in terms of the ordering $\prec_f$,

ii) if $i$ and $j$ are adjacent in a row so that $j$ is to the right of $i$, then $i\nsucc_f j$.

\noindent We let $\mathcal{T}_\la (f)$ be the set of all $f$-tableaux of shape $\la$ and $d_\la(f)=|\mathcal{T}_\la (f)|$.
\end{definition}

 \begin{definition}
For a Hessenberg function $f : [n]\rightarrow [n]$, a partition $\la\vdash m\leq n$, and a subset $A\subseteq [n]$ of size $m$, a \emph{partial $f$-tableau of shape  $\la$ with content $A$} is a filling of the diagram of $\la$ with elements in $A$ satisfying two conditions for $f$-tableaux given in Definition~\ref{def:f-tableau}.
\end{definition}

\begin{proposition}[\cite{Gasha}]\label{prop:s-expansion} For a Hessenberg function $f : [n]\rightarrow [n]$,
\begin{equation}\label{eq:Gasha}
\omega X_{G(f)}(\mathbf{x})=\sum_\la d_\la(f) s_{\lambda}(\x)\,.
\end{equation}
\end{proposition}

\begin{example} Let $f : [4]\rightarrow [4]$ be a function given by $(f(1), f(2), f(3), f(4))=(2, 3, 4, 4)$ so that $1\prec_f 3,  1\prec_f 4$, and $2\prec_f 4$. Then, there are eight $f$-tableaux of shape $(4)$, four and two $f$-tableaux of shape $(3,1)$ and $(2,2)$, respectively.
See Figure~\ref{fig:h-tableaux1}. Examples of `partial' $f$-tableaux include $\tableau{1&2\\4}$ and $\tableau{2&4}$.

\begin{figure}[ht]
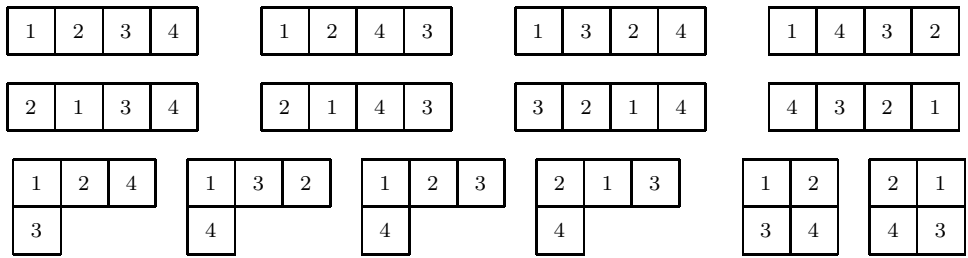

\begin{center}
$\tableau{1&2&3&4}\qquad \tableau{1&2&4&3}\qquad \tableau{1&3&2&4} \qquad\tableau{1&4&3&2}$
\bigskip

$\tableau{2&1&3&4}\qquad \tableau{2&1&4&3}\qquad \tableau{3&2&1&4} \qquad\tableau{4&3&2&1}$
\bigskip

$\tableau{1&2&4\\3}\quad \tableau{1&3&2\\4}\quad \tableau{1&2&3\\4} \quad\tableau{2&1&3\\4}    \quad\quad  \tableau{1&2\\3&4} \quad   \tableau{2&1\\4&3}$
\end{center}
 \caption{\label{fig:h-tableaux1}  $f$-tableaux for $f=(2, 3, 4, 4)$. }
\end{figure}

Hence, we have $\omega X_{G(f)}(\mathbf{x})=8s_{(4)}+4s_{(3,1)}+2s_{(2,2)}$.
\end{example}

Hessenberg functions correspond to Dyck paths in a natural way, and we define the bounce number of a Hessenberg function as a statistic of the bounce path of the corresponding Dyck path: (See \cite{Loe} for example.)

\begin{definition}  For a Hessenberg function $f : [n]\rightarrow [n]$, define a sequence as $x_1=f(1)$ and $x_{l+1}=f(x_l+1)$ for $l=1, 2, \dots$ as long as $x_l<n$.
\begin{enumerate}
\item  The \emph{Dyck path of $f$} is the path from $(0,0)$ to $(n,n)$ such that $n$ east steps are from $(i-~1,f(i))$ to $(i, f(i))$ for $i=1, 2, \dots, n$.
\item The \emph{bounce path of $f$} is the path connecting the points  $(0,0)$, $(0, x_1)$, $(x_1, x_1)$, $(x_1, x_2)$, $(x_2, x_2)$, $(x_2, x_3)$, $\dots$, $(n,n)$, vertically and horizontally alternatingly.
\item The \emph{bounce number $b(f)$ of $f$} is the number of points the bounce path of $f$ hits the diagonal line $y=x$, except the initial point $(0,0)$, i.e. the $k$ such that $x_k=n$.
\item  We let $P_l (f)=\{x_{l-1}+1, \dots, x_l\}$ for $l=1, \dots, b(f)$ where we set $x_0=0$, so that  $P_1(f), P_2 (f), \dots, P_{b(f)} (f)$ form a set partition of $[n]$.
\item We call the unit square box $\{(x, y)\,|\, i-1<x<i, j-1<y<j\}$ for $1\leq i< j \leq n$, \emph{$(i, j)$-square}.
\end{enumerate}
\end{definition}

\begin{example} If $f$ is given by $(f(1), f(2), f(3), f(4))=(2, 3, 4, 4)$, then $x_1=2$, $x_2=f(3)=4$ hence the bounce number $b(f)$ of $f$ is $2$ and $P_1(f)=\{1, 2\}, P_2(f)=\{3, 4\}$. The paths from $(0,0)$ to $(4,4)$ drawn as solid line and dashed line are the Dyck path of $f$ and the bounce path of $f$, respectively in Figure~\ref{fig:bounce path}. The $(i, j)$-squares are indicated in the figure also.

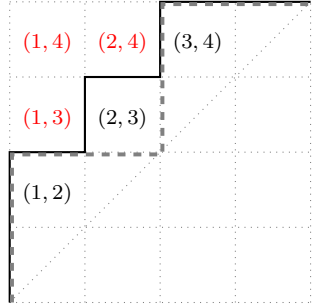
\begin{figure}[ht]
\begin{center}
    \begin{tikzpicture}
    \draw[draw=gray, dotted] (0,0) --(4,0);
     \draw[draw=gray, dotted] (0,1) --(4,1);
      \draw[draw=gray, dotted] (0,2) --(4,2);
       \draw[draw=gray, dotted] (0,3) --(4,3);
        \draw[draw=gray, dotted] (0,4) --(4,4);

    \draw[draw=gray, dotted] (0,0) --(0,4);
     \draw[draw=gray, dotted] (1,0) --(1,4);
      \draw[draw=gray, dotted] (2,0) --(2,4);
       \draw[draw=gray, dotted] (3,0) --(3,4);
        \draw[draw=gray, dotted] (4,0) --(4,4);
         \draw[draw=gray, dotted] (0,0) --(4,4);

        \draw[thick] (0,0)--(0,2)--(1,2)--(1,3)--(2,3)--(2,4)--(4,4);
         \draw[draw=gray, line width=1.5pt, dashed] (0.03,0) --(0.03,1.97)--(2.03,1.97)--(2.03, 3.97)--(3.97, 3.97);
        \node at (0.5,1.45) {\tiny$(1,2)$}; \node at (0.5,2.45) {\tiny\red$(1,3)$}; \node at (0.5,3.45) {\tiny\red$(1,4)$};
         \node at (1.5,2.45) {\tiny$(2,3)$}; \node at (1.5,3.45) {\tiny\red$(2,4)$};
         \node at (2.5,3.45) {\tiny$(3,4)$};
    \end{tikzpicture}
  \end{center}
  \caption{\label{fig:bounce path} The Dyck path and the bounce path of $f=(2, 3, 4, 4)$. }
\end{figure}
\end{example}

\begin{lemma}\label{lem:basic} Let $f$ be a Hessenberg function.
\begin{enumerate}
\item $i<j$ are incomparable with respect to the order $\prec_f$ if and only if  $(i, j)$-square is below the Dyck path of $f$.
\item  If $b\prec_f c$ in $P(f)$ and $a\leq b$, $c\leq d$, then $a\prec_f c$ and $b\prec_f d$.
\item Two elements $a, b$ in $P_l(f)$  are incomparable with respect to the order $\prec_f$ for all $l$.
\item  Let $a_1\prec_f a_2 \prec_f \cdots \prec_f a_k$ be a chain in $P(f)$ with $a_i\in P_{l(i)}(f)$, $i=1, \dots, k$. Then  $l(i) \leq l(i+1)$ for $i \leq k-1,$ and $k\leq b(f)$.  Moreover, if $a_1\prec_f a_2 \prec_f \cdots \prec_f a_{b(f)}$ is a chain, then $a_l\in P_l(f)$ for all $l=1, 2, \dots , b(f)$.
\item \emph{[(3+1)-free condition]} For a chain $a_1\prec_f a_2 \prec_f a_3$ and an element $b$ of $P(f)$, if $a_1\nprec_f b$ then $b\prec_f a_3$ and if $b\nprec_f a_3$ then $a_1\prec_f  b$.
\end{enumerate}
\end{lemma}
\begin{proof} (1), (2) and (3) are immediate from the definition of the order relation $\prec_f$, and (4) follows from  (2) and (3).  Suppose that $a_1\prec_f a_2 \prec_f a_3$ and $a_1\nprec_f b$. Then $b\leq a_2$ must hold, for otherwise
we have $a_1 \prec_f b$. This, with the condition $a_2\prec_f a_3$ implies $b\prec_f a_3$, completing the proof of a part of the (3+1)-free condition. A similar argument works for the other part of (5).
\end{proof}

\begin{remark} For a given Hessenberg function $f\,:\, [n] \rightarrow [n]$, the bounce number $b(f)$ is the same as the independence number of $G(f)$ and the length of the longest chain in $P(f)$.
\end{remark}

The following lemma is immediate from Lemma~\ref{lem:basic}, yet plays an important role in the subsequent arguments.

\begin{lemma}[Lemma 2.15 in \cite{CH}] Let $\omega X_{G(f)}(\mathbf{x})=\sum_\la d_\la(f) s_{\lambda}(\x)=\sum_\la c_\la(f) h_\la(\x)$ for a Hessenberg function  $f$. Then, $d_\la(f)=0$ for a partition $\la$ with $\ell(\la)>b(f)$ and therefore $c_\la(f)=0$ for a partition $\la$ with $\ell(\la)>b(f)$.
\end{lemma}
\begin{proof}  There cannot be a chain longer than the bounce number of $f$ due to Lemma~\ref{lem:basic} (4), and there is no $f$-tableau of shape $\la$ if $\ell(\la)>b(f)$. Hence, by Proposition~\ref{prop:s-expansion} we can conclude that  $d_\la(f)=0$ if $\ell(\la)>b(f)$. The Jacobi-Trudi identity \eqref{eq:Jacobi-Trudi} completes the proof since $h_\mu$ appears in the expansion of $s_\la$ only when $\ell(\mu)\leq \ell(\la)$.
\end{proof}

With the notion of bounce number, the main theorem (Theorem~\ref{thm:main}) of the current paper can be stated as follows.

\medskip
{\bf Main Theorem.} For a given Hessenberg function $f : [n] \rightarrow [n]$ with   $b(f)= 3$, $\omega X_{G(f)}(\mathbf{x})$ is $h$-positive.

\section{$h$-expansion of $\omega X_{G(f)}(\mathbf{x})$ when $b(f)=3$}\label{sec:bounce3}

Let us fix a positive integer $n$ and a Hessenberg function $f: [n] \rightarrow [n]$ with $b(f)=3$ and consider the expansion of  $\omega X_{G(f)}(\mathbf{x})$ into the sum of homogeneous symmetric functions $h_\mu$. We use $\mathrm{Par}(n, \leq 3)$ to denote the set of all partitions of $n$ with length at most $3$.

 If $\la=(\la_1, \la_2, \la_3)\in \mathrm{Par}(n, \leq 3)$ is a partition of $n$ with length at most $3$, allowing $0$ for $\la_2$ and $\la_3$, then Jacobi-Trudi identity \eqref{eq:Jacobi-Trudi} becomes
\begin{equation}
s_\la = \mathrm{det}\left[
  \begin{array}{lll}
     h_{\la_1} &  h_{\la_1+1} &  h_{\la_1+2} \\
     h_{\la_2-1} & h_{\la_2} &  h_{\la_2+1}  \\
     h_{\la_3-2} &  h_{\la_3-1} &  h_{\la_3}
  \end{array} \right].
\end{equation}

We define two (signed) sets associated with $\la=(\la_1, \la_2, \la_3)$:
\begin{align*}
S(\la)=\{&(\la_1, \la_2, \la_3)^+, (\la_1, \la_2+1, \la_3-1)^-,\\
& (\la_1+1, \la_2-1, \la_3)^-, (\la_1+1, \la_2+1, \la_3-2)^+,\\
&(\la_1+2, \la_2-1, \la_3-1)^+, (\la_1+2, \la_2, \la_3-2)^-\}, \mbox{ and }\\
\widetilde{S(\la)}=\{ & \widetilde{\alpha} \,|\, \alpha\in S(\la)\}\,,
\end{align*}
where $\widetilde{\alpha}$ is obtained by rearranging the parts of $\alpha$ so that the parts are non-increasing and the sign of $\widetilde{\alpha}$ is the same as the sign of $\alpha$.

Then we let $$\mathcal{S}=\bigcup_{\la\in \mathrm{Par}(n, \leq 3)} S(\la) \,\, \mbox{ and }  \,\, \widetilde{\mathcal{S}}=\bigcup_{\la\in \mathrm{Par}(n, \leq 3)} \widetilde{S(\la)}\,.$$

We use $\mathrm{sgn}(\alpha, \lambda)$ for the sign of $\alpha\in S(\la)$, which we write to the upper right of $\alpha$ for convenience in the definition of $S(\la)$. Remember that $h_\alpha=h_{\alpha_1}  h_{\alpha_2} h_{\alpha_3}$ for each $\alpha=(\alpha_1, \alpha_2, \alpha_3)\in \mathcal{S}$, where $h_0=1$ and $h_{-k}=0$ for $k$ a positive integer.

From Proposition~\ref{prop:s-expansion} and the Jacobi-Trudi identity, we have
\begin{equation}\label{eq:h-expansion}
\omega X_{G(f)}(\mathbf{x})=\sum_\la d_\la(f) s_{\lambda}(\x)\\
                       =\sum_\la d_\la(f) \sum_{\alpha\in S(\la) } \mathrm{sgn}(\alpha, \la) h_\alpha\,.
\end{equation}

Now, we let
$$\mathcal{C}_\mu=\{ \alpha\in \mathcal{S}\,|\, \widetilde{\alpha}=\mu  \}\, \mbox{ for $\mu \in \mathrm{Par}(n, \leq 3) $, and }$$

$$\mathcal{K}_\alpha=\{\la \in \mathrm{Par}(n, \leq 3)  \,|\, \alpha\in S(\la)\} \mbox{ for $\alpha\in \mathcal{S}$}\,.$$

Then \eqref{eq:h-expansion} becomes

\begin{equation}\label{eq:main}
\omega X_{G(f)}(\mathbf{x})=\sum_{\mu\in \mathrm{Par}(n, \leq 3)} \left( \sum_{\alpha\in\mathcal{C}_\mu }\,\,\sum_{\la \in \mathcal{K}_\alpha}  \mathrm{sgn}(\alpha, \la)\,d_\la(f) \right)  h_\mu \,.
\end{equation}

\bigskip

\begin{lemma}\label{lem:C_mu} For any $\mu=(\mu_1, \mu_2, \mu_3)\in \mathrm{Par}(n, \leq 3)$,
$$\mathcal{C}_\mu=
\begin{dcases}
\{(\mu_1, \mu_2, \mu_3)\} & \mbox{ if } \mu_1\not=\mu_2+1 \text{ and }\, \mu_2\not=\mu_3+1\,, \\
\{(\mu_1, \mu_2, \mu_3), (\mu_2, \mu_1, \mu_3) \} &  \text{ if } \mu_1=\mu_2+1 \text{ and }\, \mu_2\not=\mu_3+1\,, \\
\{(\mu_1, \mu_2, \mu_3), (\mu_1, \mu_3, \mu_2) \} &  \text{ if } \mu_1\not=\mu_2+1 \text{ and }\, \mu_2=\mu_3+1\,, \\
\{(\mu_1, \mu_2, \mu_3), (\mu_2, \mu_1, \mu_3), (\mu_1, \mu_3, \mu_2) \} &  \text{ if } \mu_1=\mu_2+1 \text{ and }\, \mu_2=\mu_3+1 \,.
\end{dcases}
$$
\end{lemma}
\begin{proof} We first remark that, for a partition $\la\in\mathrm{Par}(n, \leq 3)$ the only elements of $S(\la)$ that can be a \emph{non-partition} are $ (\la_1, \la_2+1, \la_3-1)$ and  $(\la_1+1, \la_2-1, \la_3)$. Moreover, $(\la_1, \la_2+1, \la_3-1)$ is not a partition if and only if $\la_1=\la_2$ and $(\la_1+1, \la_2-1, \la_3)$ is not a partition if and only if $\la_2=\la_3$. For a $\mu=(\mu_1, \mu_2, \mu_3)\in \mathrm{Par}(n, \leq 3)$, $\alpha=(\alpha_1, \alpha_2, \alpha_3)\in \mathcal{C}_\mu$ if and only if $\widetilde{\alpha}=\mu$ and $\alpha\in S(\la)$ for some $\la\in \mathrm{Par}(n, \leq 3)$.

Let $\alpha=(\alpha_1, \alpha_2, \alpha_3)$ be an element of $\mathcal{C}_\mu$ where $\alpha\in S(\la)$ for $\la\in \mathrm{Par}(n, \leq 3)$. Then, one of the following cases happens:
\begin{itemize}
\item[i.] $\alpha= (\la_1, \la_2, \la_3)=\mu.$
\item[ii.] $\alpha= (\la_1, \la_2+1, \la_3-1)=
                        \begin{dcases} (\mu_1, \mu_2, \mu_3) & if \, \la_1>\la_2,\\
                                                 (\mu_2, \mu_1, \mu_3) & if \, \la_1=\la_2.
                         \end{dcases}$
\item[iii.] $\alpha=(\la_1+1, \la_2-1, \la_3)=
                          \begin{dcases} (\mu_1, \mu_2, \mu_3) & if \, \la_2>\la_3,\\
                                                 (\mu_1, \mu_3, \mu_2) & if \, \la_2=\la_3.
                         \end{dcases}$
\item[iv.] $\alpha= (\la_1+1, \la_2+1, \la_3-2)=\mu$.
\item[v.] $\alpha= (\la_1+2, \la_2-1, \la_3-1)=\mu$.
\item[vi.] $\alpha=  (\la_1+2, \la_2, \la_3-2)=\mu$.
\end{itemize}

Therefore, we can conclude that $\alpha= (\mu_1, \mu_2, \mu_3) $ is an element of $\mathcal{C}_\mu$ for all $\mu\in  \mathrm{Par}(n, \leq 3)$, $\alpha=(\mu_2, \mu_1, \mu_3)$ is an element of  $\mathcal{C}_\mu$ if $\mu_1=\mu_2+1$, and $(\mu_1, \mu_3, \mu_2)$ is an element of $\mathcal{C}_\mu$ if $\mu_2=\mu_3+1$.

\end{proof}

\begin{lemma}\label{lem:T_a} For  $\alpha=(\alpha_1, \alpha_2, \alpha_3)\in  \mathcal{S}$, let
\begin{align*} T(\alpha)=\{ &(\alpha_1-2, \alpha_2, \alpha_3+2), (\alpha_1-2, \alpha_2+1, \alpha_3+1), \\
&(\alpha_1-1, \alpha_2-1, \alpha_3+2), (\alpha_1-1, \alpha_2+1, \alpha_3),\\
&(\alpha_1, \alpha_2-1, \alpha_3+1), (\alpha_1, \alpha_2, \alpha_3)  \}\,.
\end{align*}
Then, $\mathcal{K}_\alpha=T(\alpha)\cap \mathrm{Par}(n, \leq 3)$.
\end{lemma}
\begin{proof} For any $\la\in T(\alpha)\cap \mathrm{Par}(n, \leq 3)$, $\la\in \mathcal{K}_\alpha$ due to the way how $T(\alpha)$ is defined. Hence, we have  $T(\alpha)\cap \mathrm{Par}(n, \leq 3)\subseteq \mathcal{K}_\alpha$. Now, let $\la\in \mathrm{Par}(n, \leq 3)$ be an element of  $\mathcal{K}_\alpha$ so that $\alpha\in S(\la)$. Then, by the definitions of $S(\la)$ and $T(\alpha)$, $\la$ must be in $T(\alpha)$.
\end{proof}

\label{subsection:diagrams}

Since we have Lemma~\ref{lem:C_mu} and Lemma~\ref{lem:T_a}, we are ready to analyse the coefficient
\begin{equation}\label{eq:coefficient}
 c_\mu:=\sum_{\alpha\in\mathcal{C}_\mu }\,\,\sum_{\la \in \mathcal{K}_\alpha}  \mathrm{sgn}(\alpha, \la)\,d_\la(f)
\end{equation}
of $h_\mu$ in the expansion of $\omega X_{G(f)}(\x)$, given in \eqref{eq:main}.

There are four cases to be considered according to Lemma~\ref{lem:C_mu}. In each case, we draw a diagram of $ \cup_{\alpha\in \mathcal{C}_\mu} \mathcal{K}_\alpha$ in which two partitions $\xi, \eta \in \cup_{\alpha\in \mathcal{C}_\mu} \mathcal{K}_\alpha$ are connected by an arrow, $\xi \stackrel{\sigma_{j, i}^k}{\longrightarrow} \eta$ for $1\leq i<j\leq 3$ and $k\in\{1, 2\}$, if $\eta$ is obtained from $\xi$ by subtracting $k$  from the $j$th part of $\xi$ and adding $k$ to the $i$th  part of $\xi$. We write $\mathrm{sgn}(\alpha, \la)$ to the upper right of each $\la\in\mathcal{K}_\alpha$.
We remark that some elements in the diagram can be obsolete, i.e. can be a non-partition, depending on the given partition $\mu$. Moreover, if $\eta$ is not a partition where $\xi \stackrel{\sigma_{j, i}^k}{\longrightarrow} \eta$ in the diagram, then $\xi$ is not a partition either.

\medskip

{\bf Case I} \quad If $\mu_1\not=\mu_2+1$ and $\mu_2\not=\mu_3+1$, then $\mathcal{C}_\mu=\{\alpha=(\mu_1, \mu_2, \mu_3)\}$ and the diagram of $\mathcal{K}_\alpha$ is given in Figure~\ref{fig:diagram I}.

\begin{figure}
\begin{center}
$$\mathcal{K}_\alpha=\quad
\begin{tikzcd}[row sep=huge, column sep=tiny]
(\mu_1-2, \mu_2, \mu_3+2)^- \arrow[r,"\sigma_{3,2}^1"]\arrow[d,"\sigma_{2,1}^1"] \arrow[rrdd,"\sigma_{3,1}^2", pos=0.3 ]& (\mu_1-2, \mu_2+1, \mu_3+1)^+ \arrow[rd,"\sigma_{3,1}^1"]\arrow[dd,"\sigma_{2,1}^2", pos=0.2, crossing over] & \\
(\mu_1-1, \mu_2-1, \mu_3+2)^+\arrow[rd,"\sigma_{3,1}^1"]\arrow[rr,"\sigma_{3,2}^2", pos=0.8, crossing over]  &&(\mu_1-1, \mu_2+1, \mu_3)^-\arrow[d,"\sigma_{2,1}^1"]\\
& (\mu_1, \mu_2-1, \mu_3+1)^-\arrow[r,"\sigma_{3,2}^1"]&{\alpha=(\mu_1, \mu_2, \mu_3)^+}
\end{tikzcd}
$$
 \end{center}
  \caption{\label{fig:diagram I} The diagram of  $\cup_{\alpha\in \mathcal{C}_\mu} \mathcal{K}_\alpha$ when $\mu_1\ne\mu_2+1$ and $\mu_2 \ne \mu_3+1$. }
\end{figure}

\medskip

{\bf Case II}\quad If $\mu_1=\mu_2+1$ and $\mu_2\not=\mu_3+1$, then $\mathcal{C}_\mu=
\{\alpha=(\mu_1, \mu_2, \mu_3), \beta=(\mu_2, \mu_1, \mu_3) \}$. In this case, $ (\mu_1-2, \mu_2, \mu_3+2), (\mu_1-2, \mu_2+1, \mu_3+1), (\mu_1-1, \mu_2+1, \mu_3) \in T(\alpha)$ are not partitions, and only $(\mu_1-1, \mu_2+1, \mu_3)$ is a partition in $T(\beta)$.   Hence $\mathcal{K}_\alpha\cap \mathcal{K}_\beta=\varnothing$, and the signs with $\la$ in the diagram are for either $\mathrm{sgn}(\alpha, \la)$ or $\mathrm{sgn}(\beta, \la)$ depending whether $\la\in \mathcal{K}_\alpha$ or  $\la\in \mathcal{K}_\beta$.  The elements of $\mathcal{K}_\alpha$ are colored in red and the ones in $\mathcal{K}_\beta$ are colored in blue. Again, note that some element in the diagram can be obsolete depending on the given $\mu$. See Figure~\ref{fig:diagram II}.   We remark that $\beta$ itself is not an element of $\mathcal{K}_\alpha\cup \mathcal{K}_\beta$ since it is not a partition.

\begin{figure}[ht]
\begin{center}
$${\red \mathcal{K}_\alpha} \cup {\blue \mathcal{K}_\beta}=
\begin{tikzcd}[row sep=huge, column sep=small]
{\red(\mu_1-1, \mu_2-1, \mu_3+2)^+}\arrow[rd,"\sigma_{3,1}^1"] \arrow[r,"\sigma_{3,2}^1"] & {\blue(\mu_1-1, \mu_2, \mu_3+1)^- }\arrow[d,"\sigma_{2,1}^1"]\arrow[rd,"\sigma_{3,1}^1"]\arrow[r,dashed] &{\blue \Big(\beta=(\mu_1-1, \mu_2+1, \mu_3)\Big) }\arrow[d,dashed] \\
&{\red (\mu_1, \mu_2-1, \mu_3+1)^-}\arrow[r,"\sigma_{3,2}^1"]&{\red\alpha=(\mu_1, \mu_2, \mu_3)^+}
\end{tikzcd}
$$
 \end{center}
  \caption{\label{fig:diagram II} The diagram of  $\cup_{\alpha\in \mathcal{C}_\mu} \mathcal{K}_\alpha$ when $\mu_1=\mu_2+1$ and $\mu_2\ne\mu_3+1$. }
\end{figure}

\medskip

{\bf Case III}\quad  If $\mu_1\not=\mu_2+1$ and $\mu_2=\mu_3+1$, then $\mathcal{C}_\mu=
\{\alpha=(\mu_1, \mu_2, \mu_3), \gamma=(\mu_1, \mu_3, \mu_2) \}$.  It is easy to show that $\mathcal{K}_\alpha\cap \mathcal{K}_\gamma=\varnothing$ and the diagram of $\mathcal{K}_\alpha\cup \mathcal{K}_\gamma$ is given in Figure~\ref{fig:diagram III}. The elements of $\mathcal{K}_\alpha$ are colored red and the ones in $\mathcal{K}_\gamma$ are colored in blue.   Note that $\gamma$ itself is not an element of $\mathcal{K}_\alpha\cup \mathcal{K}_\gamma$ since it is not a partition.

\begin{figure}[ht]
\begin{center}
$${\red \mathcal{K}_\alpha}\cup {\blue \mathcal{K}_\gamma}=\quad
\begin{tikzcd}[row sep=huge]
{\red(\mu_1-2, \mu_2+1, \mu_3+1)^+}\arrow[d,"\sigma_{2,1}^1"] \arrow[rd,"\sigma_{3,1}^1"] & \\
{\blue(\mu_1-1, \mu_2, \mu_3+1)^-}\arrow[r,"\sigma_{3,2}^1"] \arrow[d, dashed]\arrow[rd,"\sigma_{3,1}^1"]  &{\red(\mu_1-1, \mu_2+1, \mu_3)^-}\arrow[d,"\sigma_{2,1}^1"]\\
{\blue \Big(\gamma=(\mu_1, \mu_2-1, \mu_3+1)\Big) }\arrow[r,dashed]&{\red\alpha=(\mu_1, \mu_2, \mu_3)^+}
\end{tikzcd}
$$
 \end{center}
  \caption{\label{fig:diagram III} The diagram of  $\cup_{\alpha\in \mathcal{C}_\mu} \mathcal{K}_\alpha$ when $\mu_1\ne\mu_2+1$ and $\mu_2=\mu_3+1$. }
\end{figure}

\medskip

{\bf Case IV}\quad  If $\mu_1=\mu_2+1$ and $\mu_2=\mu_3+1$, then $\mathcal{C}_\mu=
\{\alpha=(\mu_1, \mu_2, \mu_3), \beta=(\mu_2, \mu_1, \mu_3), \gamma=(\mu_1, \mu_3, \mu_2)  \}$. In this case, $\mathcal{K}_\beta = \mathcal{K}_\gamma=\{(\mu_1-1, \mu_2, \mu_3+1)^-\}$ and $\mathcal{K}_\alpha=\{\alpha\}$. Hence we draw a diagram for $\mathcal{K}_\alpha\cup \mathcal{K}_\beta \cup \mathcal{K}_\gamma$ as in Figure~\ref{fig:diagram IV} where $2(\mu_1-1, \mu_2, \mu_3+1)^-$ means $(\mu_1-1, \mu_2, \mu_3+1)^-$ appears in both $\mathcal{K}_\beta$ and $\mathcal{K}_\gamma$ with negative signs.   In this case, neither $\beta$ nor $\gamma$ is contained in $\mathcal{K}_\alpha\cup \mathcal{K}_\beta \cup \mathcal{K}_\gamma$ since they are not partitions.

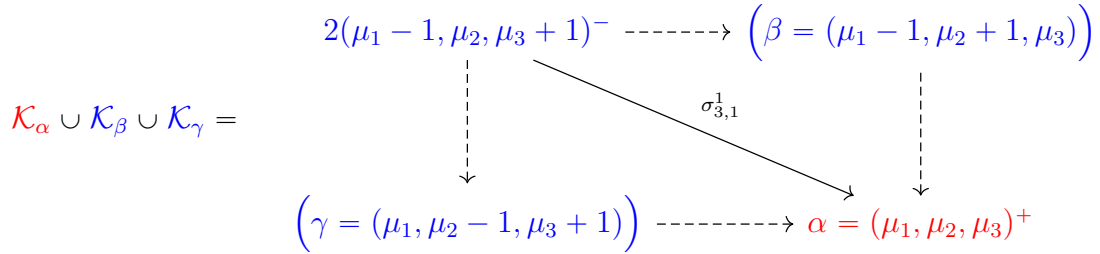
\begin{figure}[ht]
\begin{center}
$${\red \mathcal{K}_\alpha}\cup {\blue \mathcal{K}_\beta} \cup {\blue \mathcal{K}_\gamma}=\quad
\begin{tikzcd}[row sep=huge]
{\blue 2(\mu_1-1, \mu_2, \mu_3+1)^-}\arrow[r,dashed]\arrow[d,dashed]\arrow[rd,"\sigma_{3,1}^1"] &{\blue \Big(\beta=(\mu_1-1, \mu_2+1, \mu_3)\Big) }\arrow[d,dashed] \\
{\blue \Big(\gamma=(\mu_1, \mu_2-1, \mu_3+1)\Big) }\arrow[r,dashed] &{\red\alpha=(\mu_1, \mu_2, \mu_3)^+}
\end{tikzcd}
$$
 \end{center}
  \caption{\label{fig:diagram IV} The diagram of  $\cup_{\alpha\in \mathcal{C}_\mu} \mathcal{K}_\alpha$ when $\mu_1=\mu_2+1$ and $\mu_2=\mu_3+1$. }
\end{figure}

\section{Proof of $h$-positivity when $b(f)= 3$}\label{sec:h-positivity}

In this section we prove Theorem~\ref{thm:main}, that is the Stanley-Stembridge conjecture is true when $f$ has bounce number  $3$. We cancel the negative terms in \eqref{eq:coefficient} with positive terms so that the remaining terms have positive signs.

Since
\begin{equation}
 \omega X_{G(f)}(\mathbf{x})=\sum_{\mu\in \mathrm{Par}(n, \leq 3)} \left( \sum_{\alpha\in\mathcal{C}_\mu }\,\,\sum_{\la \in \mathcal{K}_\alpha}  \mathrm{sgn}(\alpha, \la)\,d_\la(f) \right)  h_\mu= \sum_{\mu\in \mathrm{Par}(n, \leq 3)} c_\mu\, h_\mu
\end{equation}
and we classified $\mathcal{C}_\mu$ and $ \cup_{\alpha\in \mathcal{C}_\mu} \mathcal{K}_\alpha$ depending on the given $\mu$ in Section~\ref{sec:bounce3}, 
we consider each of the four cases I - IV separately and show that every negative term in
\begin{equation}\label{eq:c_mu}
c_{\mu}=\sum_{\alpha\in\mathcal{C}_\mu }\sum_{\la \in \mathcal{K}_\alpha}  \mathrm{sgn}(\alpha, \la)\,d_\la(f)\end{equation}
can be canceled with a positive term.

Remember that $d_\la(f)=|\mathcal{T}_\la (f)|$ is the number of $f$-tableaux of shape $\la$. We fix a positive integer $n$ and a Hessenberg function $f: [n] \rightarrow [n]$ with bounce number $b(f)=3$,  and we use $\mathcal{T}(\la)$, $d(\la)$, $P$ and $P_l$, $l=1, 2, 3$, instead of $\mathcal{T}_\la (f)$, $d_\la(f)$, $P(f)$ and $P_l(f)$, respectively for convenience. We also use $a\prec b$ instead of $a\prec_f b$ for the order relation in $P=P_1\cup P_2\cup P_3$.

The basic idea is to use the relation $\xi \stackrel{\sigma_{j, i}^k}{\longrightarrow} \eta$ in the diagram of $\cup_{\alpha\in \mathcal{C}_\mu} \mathcal{K}_\alpha$ to define an injective map from each negative $\mathcal{T}(\xi)$ to a positive $\mathcal{T}(\eta)$, where the sign of $\mathcal{T}(\lambda)$ for $\lambda\in \cup_{\alpha\in \mathcal{C}_\mu} \mathcal{K}_\alpha$ is the sign $\mathrm{sgn}(\alpha, \la)$ of $\lambda$ for a corresponding $\alpha\in \mathcal{C}_\mu$.

\begin{definition} Let $\xi=(\xi_1, \xi_2, \xi_3)$ be a partition and $T$ be a tableau of shape $\xi$. For given $i, j$ such that  $1\leq i<j\leq 3$,
\begin{enumerate}
\item if $\xi_j\geq 1$, we define $\sigma_{j\rightarrow i}^\square (T)$ to be the tableau  obtained by moving the rightmost entry of the $j$th row of $T$ to the end of the $i$th row.
\item if $\xi_j\geq 2$, we define $\sigma_{j\rightarrow i}^{\square\!\square} (T)$ to be the tableau obtained by moving two rightmost entries of the $j$th row of $T$ to the end of the $i$th row.
\end{enumerate}
\end{definition}

We need to mention that $\sigma_{j\rightarrow i}^*$, $*\in\{\square,  \square\!\square\}$, does not necessarily send an $f$-tableau to an $f$-tableau, but $\sigma_{3\rightarrow 1}^*$ always does.

 \begin{lemma}\label{lem:3to1} Let $\xi=(\xi_1, \xi_2, \xi_3)$ be a partition and  $T$ be an $f$-tableau in $\mathcal{T}(\xi)$.
 \begin{enumerate}
     \item If $\xi_3\geq 1$ then $\sigma_{3\rightarrow 1}^\square (T)$ is an $f$-tableau in  $\mathcal{T}(\xi_1+1, \xi_2, \xi_3-1)$. Hence,
 $\sigma_{3\rightarrow 1}^\square$ is a map from $\mathcal{T}(\xi_1, \xi_2, \xi_3)$ to $\mathcal{T}(\xi_1+1, \xi_2, \xi_3-1)$ and it is injective.
 \item If $\xi_3\geq 2$ then $\sigma_{3\rightarrow 1}^{\square\!\square} (T)$ is an $f$-tableau in  $\mathcal{T}(\xi_1+2, \xi_2, \xi_3-2)$. Hence,
 $\sigma_{3\rightarrow 1}^{\square\!\square}$ is a map from $\mathcal{T}(\xi_1, \xi_2, \xi_3)$ to $\mathcal{T}(\xi_1+2, \xi_2, \xi_3-2)$ and it is injective.
 \end{enumerate}
 \end{lemma}
 \begin{proof} For  $T\in\mathcal{T}(\xi_1, \xi_2, \xi_3)$, the entries in the third row of $T$ are in $P_3$ and  if $a\in P_3$ then  $x\nsucc a$ for any $x$. Hence $\sigma_{3\rightarrow 1}^\square (T)$ must be an $f$-tableau in $\mathcal{T}(\xi_1+1, \xi_2, \xi_3-1)$. The injectivity of $\sigma_{3\rightarrow 1}^\square$ is immediate from the definition. Almost the same argument works for the second part.
 \end{proof}

\begin{example}\label{ex:23567888}
 Let $f : [8]\rightarrow [8]$ be a Hessenberg function given by $$(f(1), f(2), f(3), f(4), f(5), f(6), f(7), f(8))=(2, 3, 5, 6, 7, 8, 8, 8)$$ so that the bounce number $b(f)$ of $f$ is $3$ and $P_1=\{1, 2\}, P_2=\{3, 4, 5\}, P_3=\{6, 7, 8\}$. If we let $T$ be an $f$-tableau of shape $\xi=(4, 2, 2)$ given as  $$T=\tableau{2&1&5&6\\4&3\\8&7}\,, $$ then  $$\sigma_{3\rightarrow 1}^\square (T)= \tableau{2&1&5&6&{\red 7}\\4&3\\8}\,\,\in\mathcal{T}(5, 2, 1)$$ and $$\sigma_{3\rightarrow 1}^{\square\!\square} (T)= \tableau{2&1&5&6&\red{8}&\red{7}\\4&3}\,\,\in\mathcal{T}(6, 2, 0)\,.$$
We note that, since $7$ and $8$ are elements of $P_3$, $6\nsucc 7$ and $6\nsucc 8$. On the contrary, $$\sigma_{3\rightarrow 2}^\square (T)= \tableau{2&1&5&6\\4&3&{\red 7}\\8}\,\,\not\in\mathcal{T}(4, 3, 1)\,,$$
since $5\nprec 7$.
\end{example}

 \begin{lemma}\label{lem:RS} For a Hessenberg function $f: [n] \rightarrow [n]$, let $R$ be a partial $f$-tableau of shape $(m, m, m)$ with content $A\subseteq [n]$ and $S$ be a partial $f$-tableau of shape $\la=(\la_1, \la_2, \la_3)$ with content $[n]-A$. Then the tableau $T:=R\cup S$ of shape $(m+\la_1, m+\la_2, m+\la_3)$ obtained by concatenating $R$ and $S$ so that the first $m$ columns of $T$ is $R$ and the rest is the same as $S$, is an $f$-tableau.
\end{lemma}
\begin{proof}
We only need to check the conditions between the $m$th column and the $(m+1)$th column of $T$; that is, $r_l\nsucc s_l$ for $l=1, 2, 3$, where $r_l, s_l$ are the $l$th entry of the last column of $R$ and the first column of $S$, respectively. We remark that $S$ can be a tableau with one or two rows, in which case we need to check only one or two relations.  Since Lemma~\ref{lem:basic} (4) implies that $r_l\in P_l$ for $l=1, 2, 3,$ and  $s_l$ is in $P_{l'}$ for $l'\geq l$, we can conclude that  $r_l\nsucc s_l$ for $l=1, 2, 3$.
\end{proof}
In the rest of this section, we prove that $c_{\mu}=\sum_{\alpha\in\mathcal{C}_\mu }\sum_{\la \in \mathcal{K}_\alpha}  \mathrm{sgn}(\alpha, \la)d (\lambda)$ is nonnegative for any partition $\mu$ by injectively mapping the elements of negative $\mathcal{T}(\xi)$ into positive  $\mathcal{T}(\eta)$ for $\xi, \eta\in \cup_{\alpha\in \mathcal{C}_\mu} \mathcal{K}_\alpha$.
We consider four separate cases as we did in Section~\ref{sec:bounce3} 
depending on the conditions that a given partition $\mu=(\mu_1, \mu_2, \mu_3)$ satisfy.
We use $A\simeq B$ when two sets $A$ and $B$ are in bijection.

\medskip

\subsection{\textbf{Case I}}


\quad  Assume that $\mu_1\not=\mu_2+1$ and $\mu_2\not=\mu_3+1$. Then from the diagram, Figure~\ref{fig:diagram I}, of $\mathcal{K}_\alpha$, we can see that we need to define an injective map
$$\mbox{ from } \mathcal{T}(\mu_1-2, \mu_2, \mu_3+2)^-\cup\mathcal{T}(\mu_1-1, \mu_2+1, \mu_3)^-\cup  \mathcal{T}(\mu_1, \mu_2-1, \mu_3+1)^-$$
$$\mbox{ into } \mathcal{T}(\mu_1-2, \mu_2+1, \mu_3+1)^+\cup \mathcal{T}(\mu_1-1, \mu_2-1, \mu_3+2)^+\cup \mathcal{T}(\mu_1, \mu_2, \mu_3)^+ \,.$$

We first use $\sigma_{3\rightarrow 1}^*$'s that were shown to be injective in Lemma~\ref{lem:3to1}:

\begin{figure}[ht]
\begin{center}
$$
\begin{tikzcd}[row sep=huge, column sep=tiny]
\mathcal{T}(\mu_1-2, \mu_2, \mu_3+2)^- \arrow[rrdd,"\sigma_{3\rightarrow 1}^{\square\!\square} ", hook] & \mathcal{T}(\mu_1-2, \mu_2+1, \mu_3+1)^+ \arrow[rd,"\sigma_{3\rightarrow 1}^\square", hook] & \\
\mathcal{T}(\mu_1-1, \mu_2-1, \mu_3+2)^+\arrow[rd,"\sigma_{3\rightarrow 1}^\square", hook] &&\mathcal{T}(\mu_1-1, \mu_2+1, \mu_3)^-\arrow[d, dashed]\\
& \mathcal{T}(\mu_1, \mu_2-1, \mu_3+1)^-\arrow[r,dashed]&{\mathcal{T}(\mu_1, \mu_2, \mu_3)^+}
\end{tikzcd}
$$
 \end{center}
  \caption{\label{fig:CaseI} Case I. }
\end{figure}

We have, as one can see in Figure ~\ref{fig:CaseI},
$$\mathcal{T}(\mu_1-2, \mu_2, \mu_3+2)^-\simeq \sigma_{3\rightarrow 1}^{\square\!\square}(\mathcal{T}(\mu_1-2, \mu_2, \mu_3+2)^-)\subseteq \mathcal{T}(\mu_1, \mu_2, \mu_3)^+ \,, $$
$$\mathcal{T}(\mu_1-1, \mu_2-1, \mu_3+2)^+ \simeq  \sigma_{3\rightarrow 1}^\square(\mathcal{T}(\mu_1-1, \mu_2-1, \mu_3+2)^+)\subseteq \mathcal{T}(\mu_1, \mu_2-1, \mu_3+1)^- \,,\,\, \mbox{  and}$$
$$\mathcal{T}(\mu_1-2, \mu_2+1, \mu_3+1)^+ \simeq  \sigma_{3\rightarrow 1}^\square(\mathcal{T}(\mu_1-2, \mu_2+1, \mu_3+1)^+)\subseteq \mathcal{T}(\mu_1-1, \mu_2+1, \mu_3)^-\,.$$

Therefore, if we let
\begin{align*}
    \widetilde{\mathcal{T}}(\mu_1, \mu_2, \mu_3)^+&= \mathcal{T}(\mu_1, \mu_2, \mu_3)^+ -  \sigma_{3\rightarrow 1}^{\square\!\square}(\mathcal{T}(\mu_1-2, \mu_2, \mu_3+2)^-)  \,,\\
\widetilde{\mathcal{T}}(\mu_1, \mu_2-1, \mu_3+1)^-&= \mathcal{T}(\mu_1, \mu_2-1, \mu_3+1)^- -\sigma_{3\rightarrow 1}^\square(\mathcal{T}(\mu_1-1, \mu_2-1, \mu_3+2)^+)\,, \\
\widetilde{\mathcal{T}}(\mu_1-1, \mu_2+1, \mu_3)^-&=\mathcal{T}(\mu_1-1, \mu_2+1, \mu_3)^- - \sigma_{3\rightarrow 1}^\square(\mathcal{T}(\mu_1-2, \mu_2+1, \mu_3+1)^+)\,,\end{align*}
we are left to define an injection
$$\phi : \widetilde{\mathcal{T}}(\mu_1-1, \mu_2+1, \mu_3)^-\cup  \widetilde{\mathcal{T}}(\mu_1, \mu_2-1, \mu_3+1)^- \longrightarrow \widetilde{\mathcal{T}}(\mu_1, \mu_2, \mu_3)^+ \,.$$

Note that if $\mu_1=\mu_2$ then  $\widetilde{\mathcal{T}}(\mu_1-1, \mu_2+1, \mu_3)^-=\varnothing$, and if $\mu_2=\mu_3$ then $\widetilde{\mathcal{T}}(\mu_1, \mu_2-1, \mu_3+1)^-=\varnothing$, and these cases will be covered as special cases of the case $\mu_1> \mu_2+1$ and $\mu_2>\mu_3+1$, respectively. We thus assume that $\mu_1> \mu_2+1$ and $\mu_2>\mu_3+1$.
A tableau $T$ in $ \widetilde{\mathcal{T}}(\mu_1-1, \mu_2+1, \mu_3)^-$ can be written as $T=R\cup S$ where $R\in \mathcal{T}(\mu_3, \mu_3, \mu_3)$ and $S \in \mathcal{T}(\mu_1-\mu_3-1, \mu_2-\mu_3+1, 0)$; see Lemma~\ref{lem:RS}. Since $\mu_1> \mu_2+1>\mu_3+2$  and we will manipulate $S$ only to define the image $\phi(T)$ of $T$, we may assume that $(\mu_1-1, \mu_2+1, \mu_3)=(m+k, m, 0)$ for $m\geq 3$ and $k\geq 0$, and therefore $(\mu_1, \mu_2-1, \mu_3+1)=(m+k+1, m-2, 1)$ for $m\geq 3$ and $k\geq 0$.
We remark that this process works as required as long as $\phi(S)$ is an $f$-tableau due to Lemma~\ref{lem:RS}.

\begin{proof}[Proof of Theorem~\ref{thm:main} in \textbf{Case I}]
In what follows, to prove Theorem~\ref{thm:main}, we will do the following steps. See Figure~\ref{fig:CaseI_reduce}.
\begin{enumerate}
\item We modify ${\sigma}_{2\rightarrow 1}^\square$ and ${\sigma}_{3\rightarrow 2}^\square$ to define injections
$$\widetilde{\sigma}_{2\rightarrow 1}^\square: \widetilde{\mathcal{T}}(m+k, m, 0)^- \rightarrow \widetilde{\mathcal{T}}(m+k+1, m-1, 0)^+\,, \text{ and }$$
$$\widetilde{\sigma}_{3\rightarrow 2}^\square: \widetilde{\mathcal{T}}(m+k+1, m-2, 1)^- \rightarrow \widetilde{\mathcal{T}}(m+k+1, m-1, 0)^+\,. $$
\item We then modify $\widetilde{\sigma}_{3\rightarrow 2}^\square$ to define $\phi_2: \widetilde{\mathcal{T}}(m+k+1, m-2, 1)^- \rightarrow \widetilde{\mathcal{T}}(m+k+1, m-1, 0)^+ $ so that the map $\phi : \widetilde{\mathcal{T}}(m+k, m, 0)^- \cup  \widetilde{\mathcal{T}}(m+k+1, m-2, 1)^- \longrightarrow \widetilde{\mathcal{T}}(m+k+1, m-1, 0)^+ $ defined by  $\phi|_{\widetilde{\mathcal{T}}(m+k, m, 0)^-}=\widetilde{\sigma}_{2\rightarrow 1}^\square:=\phi_1$ and $\phi|_{\widetilde{\mathcal{T}}(m+k+1, m-2, 1)^-}=\phi_2$ is an injection; that is, $\phi_1(\widetilde{\mathcal{T}}(m+k, m, 0)^-)\cap\phi_2(\widetilde{\mathcal{T}}(m+k+1, m-2, 1)^-)  =\varnothing\,. $
\end{enumerate}

In this subsection we will only give the definition of the maps $\widetilde{\sigma}_{2\rightarrow 1}^\square$, $\widetilde{\sigma}_{3\rightarrow 2}^\square$, and $\phi_2$ in Definition \ref{def:sigma21}, Definition \ref{def:sigma32} and Definition \ref{def:phi2}, respectively, and the proofs that they satisfy the desired properties will be done later in a new section. The proofs of Lemmas \ref{lem:sigma21} through \ref{lem:disjoint from T1} and Proposition \ref{prop:phi2}, that are rather technical will be done in Section~\ref{sec:proofs}.

\begin{figure}[ht]
\begin{center}
$$
\begin{tikzcd}[row sep=huge, column sep=huge]
\qquad\qquad & \widetilde{\mathcal{T}}(m+k, m, 0)^- \arrow[d,"\phi_1=\widetilde{\sigma}_{2\rightarrow 1}^\square", hook]\\
 \widetilde{\mathcal{T}}(m+k+1, m-2, 1)^- \arrow[r,"\phi_2={(\widetilde{\sigma}_{3\rightarrow 2}^\square)^*}", hook]&  \widetilde{\mathcal{T}}(m+k+1, m-1, 0)^+
\end{tikzcd}
$$
 \end{center}
  \caption{\label{fig:CaseI_reduce} Reduced Case I. }
\end{figure}

\medskip

\begin{definition}\label{def:sigma21}
Let $m\geq 3$ and $k\geq 0$ be integers. Then we define

$$\widetilde{\sigma}_{2\rightarrow 1}^\square : \widetilde{\mathcal{T}}(m+k, m, 0)^- \rightarrow \widetilde{\mathcal{T}}(m+k+1, m-1, 0)^+ $$  as follows. Let

$$T=\scriptsize{\ttableau{ a_1 & b_1^{(m-3)} & \cdots &b_1^{(2)}& b_1^{(1)} &b_1& c_1&   d_1^{(k-1)}&\cdots& d_1^{(2)}&d_1^{(1)}&d_1\\
                                            a_2& b_2^{(m-3)} &  \cdots &b_2^{(2)} & b_2^{(1)} & b_2& c_2} } $$

be an $f$-tableau in $ \widetilde{\mathcal{T}}(m+k, m, 0)^- \,.$ Then $a_2 \not\prec d_1$ since $T\not\in\sigma_{3\rightarrow 1}^\square(\mathcal{T}(m+k-1, m, 1)^+)$.

\begin{enumerate}
\item[$\langle 1\rangle$] When $d_1\nsucc c_2$, move $c_2$ to the end of the first row;
$$\widetilde{\sigma}_{2\rightarrow 1}^\square(T):=
\scriptsize{\ttableau{a_1& b_1^{(m-3)} & \cdots &b_1^{(2)}& b_1^{(1)} &b_1& c_1& d_1^{(k-1)}&\cdots& d_1^{(2)}&d_1^{(1)}&d_1& \red{c_2}\\
                                            a_2& b_2^{(m-3)} &  \cdots &b_2^{(2)} & b_2^{(1)} & b_2}}\,. $$
\item [$\langle 2\rangle$]  When $d_1\succ c_2$, we consider two sequences of entries  from $T$:
$$d_1^{(0)}:=d_1, d_1^{(1)}, \,d_1^{(2)}, \,\cdots, \,d_1^{(k-1)}\,, d_1^{(k)}:=c_1 \mbox{ of length $k$, in the first row and }$$
$$b_2^{(-1)}:=c_2, b_2^{(0)}:=b_2,\, b_2^{(1)},\, \cdots, \,b_2^{(m-2)}:=a_2\, \mbox{ of length $m-1$ in the second row: }$$

\ytableausetup{mathmode, boxsize=2.7em}
$$T={\SMALL \begin{ytableau}
  a_1 &  b_1^{(m-3)} & \cdots &b_1^{(2)}& b_1^{(1)} &b_1& *(green) c_1& *(green) d_1^{(k-1)}&*(green)\cdots& *(green)d_1^{(2)}&*(green)d_1^{(1)}&*(green)d_1\\
*(orange) a_2& *(orange) b_2^{(m-3)} &  *(orange) \cdots &*(orange)b_2^{(2)} & *(orange)b_2^{(1)} & *(orange) b_2& *(orange) c_2
\end{ytableau}}$$

 \begin{enumerate}

\item[$\langle 2$-$i \rangle$] If there is  $0 \leq i\leq \min(k-1, m-2)$ such that $d_1^{(j+1)}\succ b_2^{(j)}$ for all $j=-1, 0, \dots, i-1$ but $d_1^{(i+1)}\nsucc b_2^{(i)}$,  then we exchange $(i+1)$ entries in the tail of the first row and $(i+2)$ entries in the tail of the second row of $T$ to obtain

\ytableausetup{mathmode, boxsize=2.7em}
$$\widetilde{\sigma}_{2\rightarrow 1}^\square(T):=
 {\SMALL\begin{ytableau}
a_1 & \cdots & b_1^{(i+1)} & b_1^{(i)} & \cdots & b_1^{(1)} & b_1 & *(green) c_1 & *(green)\cdots & *(green) d_1^{(i+1)} & *(orange){b_2^{(i)}} & *(orange){ \cdots} & *(orange){b_2^{(0)}} &  *(orange){c_2}\\
*(orange)a_2 & *(orange)\cdots & *(orange)b_2^{(i+1)} & *(green){d_1^{(i)}} & *(green){\cdots} &  *(green){d_1^{(1)}}& *(green){d_1}
\end{ytableau}}
\,.$$

\item[$\langle 2$-$\infty \rangle$]
Otherwise, in which case we have $m-2< k-1$ and $d_1^{(j+1)} \succ b_2^{(j)}$ for all $j=0,..., m-2 $, we let
$$\widetilde{\sigma}_{2\rightarrow 1}^\square(T):=
 {\SMALL\begin{ytableau} a_1&\cdots&b_1^{(2)}&b_1^{(1)}&b_1& *(green)c_1& *(green)\cdots &*(green) d_1^{(m)}& *(yellow)\blue{d_1}&*(orange){a_2}&*(orange){b_2^{(m-3)}}&*(orange){ \cdots} &*(orange){b_2^{(1)}}& *(orange){b_2}&*(orange){c_2}\\
*(green){d_1^{(m-2)}}&*(green){\cdots} & *(green){d_1^{(2)}} &  *(green){d_1^{(1)}}& *(yellow)\blue{d_1^{(m-1)}}
\end{ytableau}}\,. $$
\end{enumerate}

\end{enumerate}
\end{definition}

\begin{remark} In Definition~\ref{def:sigma21} $\langle 2\rangle$, if $m-2\geq k-1$ then $\langle 2$-$i\rangle$ {\rm (}$0 \leq i \leq  k-1${\rm )} is always the case. That is because $c_1\in P_1 \cup P_2$ and $b_2^{(j)}\in P_2\cup P_3$ for $j= 0, \dots, m-2$  due to Lemma~\ref{lem:basic}, and hence we have $c_1=d_1^{(k)}\nsucc b_2^{(j)}$ for all $j$.
\end{remark}

\begin{remark} \label{rmk:T 2 infty} If $T$ is of type $\langle 2$-$\infty\rangle$, then $c_2 \prec d_1^{(m-1)}$.
  For, from $a_2 \not \prec d_1$  that is because $T\not\in \sigma_{3\rightarrow 1}^\square(\mathcal{T}(m+k-1, m, 1)^+) $,  and $a_2=b_1^{(m-2)} \prec d_1^{(m-1)}$  it follows that $d_1 < d_1^{(m-1)}$. Since $c_2 \prec d_1$, we have $c_2 \prec d_1^{(m-1)}$.
\end{remark}

\begin{definition}\label{def:sigma32} Let $m\geq 3$ and $k\geq 0$ be integers. Then we define

$$\widetilde{\sigma}_{3\rightarrow 2}^\square : \widetilde{\mathcal{T}}(m+k+1, m-2, 1)^- \rightarrow \widetilde{\mathcal{T}}(m+k+1, m-1, 0)^+ $$ as follows. Let

 $$S=\scriptsize{ \ttableau{a_1& b_1^{(m-3)}& \cdots&b_1^{(2)}&b_1^{(1)}& b_1&c_1& d_1^{(k-1)}& \cdots&d_1^{(2)}& d_1^{(1)}&d_1& e_1\\
                                            a_2& b_2^{(m-3)}& \cdots&b_2^{(2)}& b_2^{(1)}\\
                                            a_3}}$$
be an $f$-tableau in $\widetilde{\mathcal{T}}(m+k+1, m-2, 1)^- \,.$ Then, since $S\not\in \sigma_{3\rightarrow 1}^\square(\mathcal{T}(m+k-2, m, 2)^+)$ we have $b_2^{(m-3)} \not\prec e_1$. \\

\begin{enumerate}
\item[(1)] When $b_1 \prec a_3$, we let
         $$\widetilde{\sigma}_{3\rightarrow 2}^\square(S):=
         \scriptsize{ \ttableau{a_1& b_1^{(m-3)}& \cdots&b_1^{(2)}&b_1^{(1)}& b_1&c_1& d_1^{(k-1)}& \cdots&d_1^{(2)}& d_1^{(1)}&d_1& e_1\\
                                            a_2& b_2^{(m-3)}& \cdots&b_2^{(2)}& b_2^{(1)}&\red{a_3}
                                            }}\,.$$

\item[(2)] When $b_1 \not\prec a_3$ and $b_1 \not \succ b_2^{(m-3)}$, we let
         $$\widetilde{\sigma}_{3\rightarrow 2}^\square(S):=
         \scriptsize{ \ttableau{a_1& b_1^{(m-3)}& \cdots&b_1^{(2)}&b_1^{(1)}&\red{a_2}& c_1&d_1^{(k-1)}& \cdots& d_1^{(2)}& d_1^{(1)}&d_1& e_1\\
                                            \red{b_1}&b_2^{(m-3)}&\cdots&b_2^{(2)} & b_2^{(1)}& \blue{ a_3} }}\,.$$

\item[(3)] When $b_1 \not\prec a_3$ and $b_1   \succ b_2^{(m-3)}$,
 \begin{enumerate}
 \item[(3-1)]if $e_1 \not\prec a_3$ or ($e_1 \prec a_3$ and $a_2 \prec d_1$), then we let
        $$\widetilde{\sigma}_{3\rightarrow 2}^\square(S):=
         \scriptsize{ \ttableau{a_1& b_1^{(m-3)}&\cdots&b_1^{(2)}&b_1^{(1)}&\red{a_2}& c_1& d_1^{(k-1)}&\cdots& d_1^{(2)}& d_1^{(1)}&d_1&\red{b_1}\\
                                            \red{e_1}&b_2^{(m-3)}&\cdots&b_2^{(2)} & b_2^{(1)}& \blue{a_3} }}\,.$$
 \item[(3-2)] if $e_1 \prec a_3$ and $a_2 \not \prec d_1$, then we consider two sequences of entries from $S$:
$$d_1^{(-1)}:=e_1,\, d_1^{(0)}:=d_1, \,d_1^{(1)}, \,d_1^{(2)}, \,\cdots, \,d_1^{(k-1)}\,, d_1^{(k)}:=c_1 \mbox{ of length $k+2$, and }$$
$$b_1^{(0)}:=b_1, \,b_1^{(1)},\, b_1^{(2)},\, \cdots, \,b_1^{(m-2)}:=a_1\, \mbox{ of length $m-1$. }$$
 $$S= {\SMALL\begin{ytableau} *(yellow) a_1& *(yellow)b_1^{(m-3)}& *(yellow)\cdots & *(yellow)b_1^{(2)}&*(yellow)b_1^{(1)}&*(yellow) b_1&*(pink)c_1& *(pink)d_1^{(k-1)}& *(pink)\cdots&*(pink)d_1^{(2)}& *(pink)d_1^{(1)}&*(pink)d_1& *(pink)e_1\\
                                            a_2& b_2^{(m-3)}& \cdots&b_2^{(2)}& b_2^{(1)}\\ a_3 \end{ytableau}}$$
Let $i$ be the smallest such that $d_1^{(i)}\nprec b_1^{(i+2)}$, then we exchange the tails of length $(i+2)$ of two sequences (in the tableau) and move $a_3$ to the second row to obtain
$$\!\!\!\!\!\!\!\!\!\!\widetilde{\sigma}_{3\rightarrow 2}^\square(S):=
         {\SMALL\begin{ytableau} *(yellow)a_1& *(yellow)\cdots & *(yellow)b_1^{(i+2)}&*(pink)\red{d_1^{(i)}}&*(pink)\red{\cdots}&*(pink)\red{d_1^{(0)}}&*(pink)\red{e_1}&*(pink) c_1&*(pink)\cdots &*(pink)d_1^{(i+1)}&*(yellow)\red{b_1^{(i+1)}}&*(yellow)\red\cdots& *(yellow)\red{b_1^{(1)}}&*(yellow)\red{b_1}\\
                                           a_2&\cdots& b_2^{(i+2)}&b_2^{(i+1)} &\cdots& b_2^{(1)}& \red{a_3} \end{ytableau}}\,.$$

\end{enumerate}
\end{enumerate}
\end{definition}

\begin{remark} We note that the smallest $i$ such that $d_1^{(i)}\nprec b_1^{(i+2)}$ exists in case (3-2) of Definition~\ref{def:sigma32}: Since $a_1\prec a_2 \prec a_3$ is a chain we know that $a_1\in P_1$ and this implies  $d_1^{(j)}\nprec b_1^{(m-2)}=a_1$ for all $-1\leq j \leq k$. Moreover, since $ b_1^{(m-3)}\prec b_2^{(m-3)}\prec b_1$ is a chain we know that $b_1\in P_3$, and this implies that $c_1$ is in $P_2\cup P_3$ because $b_1 \nsucc c_1 $. Now we can conclude that  $d_1^{(k)}=c_1\nprec b_1^{(j)}$ for all $0\leq j\leq m-2$.
\end{remark}

Proofs of the following lemmas are given in Section~\ref{sec:proofs}.

\begin{lemma} \label{lem:sigma21}
$\widetilde{\sigma}_{2\rightarrow 1}^\square : \widetilde{\mathcal{T}}(m+k, m, 0)^- \rightarrow \widetilde{\mathcal{T}}(m+k+1, m-1, 0)^+ $ is a well defined injective map.
\end{lemma}

\begin{lemma} \label{lem:sigma32}
$\widetilde{\sigma}_{3\rightarrow 2}^\square : \widetilde{\mathcal{T}}(m+k+1, m-2, 1)^- \rightarrow \widetilde{\mathcal{T}}(m+k+1, m-1, 0)^+ $ is a well defined injective map.
\end{lemma}

\medskip

We now have two injective maps $\widetilde{\sigma}_{2\rightarrow 1}^\square : \widetilde{\mathcal{T}}(m+k, m, 0)^- \rightarrow \widetilde{\mathcal{T}}(m+k+1, m-1, 0)^+ $ and
$\widetilde{\sigma}_{3\rightarrow 2}^\square : \widetilde{\mathcal{T}}(m+k+1, m-2, 1)^- \rightarrow \widetilde{\mathcal{T}}(m+k+1, m-1, 0)^+ $. However,  $\widetilde{\sigma}_{2\rightarrow 1}^\square( \widetilde{\mathcal{T}}(m+k, m, 0)^-)$ and $\widetilde{\sigma}_{3\rightarrow 2}^\square(\widetilde{\mathcal{T}}(m+k+1, m-2, 1)^-)$ may intersect. We hence let  $\phi_1:=\widetilde{\sigma}_{2\rightarrow 1}^\square$ and then modify $\widetilde{\sigma}_{3\rightarrow 2}^\square$ to define $\phi_2: \widetilde{\mathcal{T}}(m+k+1, m-2, 1)^- \rightarrow \widetilde{\mathcal{T}}(m+k+1, m-1, 0)^+ $ so that $$\phi_1(\widetilde{\mathcal{T}}(m+k, m, 0)^-)\cap\phi_2(\widetilde{\mathcal{T}}(m+k+1, m-2, 1)^-)  =\varnothing\,.$$

\medskip
We give the definition of $\phi_2$ in Definition~\ref{def:phi2} after we do some necessary background work.

We divide $\phi_1(\widetilde{\mathcal{T}}(m+k, m, 0)^-)$ into two parts according to the properties of the pre-images: When we adopt the names for the entries of $T\in \widetilde{\mathcal{T}}((m+k, m, 0)^-)$ as in Definition~\ref{def:sigma21}, let
$$\widetilde{\mathcal{T}}^{+,1}=\{ \phi_1(T) \,|\, d_1\nsucc c_2 \mbox{ in } T\in \widetilde{\mathcal{T}}((m+k, m, 0)^-) \}\,,$$
$$\widetilde{\mathcal{T}}^{+,2}=\{ \phi_1(T) \,|\, d_1\succ c_2 \mbox{ in } T\in \widetilde{\mathcal{T}}((m+k, m, 0)^-) \}$$ be the images of the sets of tableaux satisfying the conditions $\langle 1\rangle$ and $\langle 2\rangle$ in Definition~\ref{def:sigma21}, respectively. We also let for $0\leq i \leq \min(k-1, m-2)$,
$$\widetilde{\mathcal{T}}^{+,2(i)}:=\{\phi_1(T)\in \widetilde{\mathcal{T}}^{+,2}\,|\,d_1^{(j+1)}\succ b_2^{(j)} \mbox{ for all } j= -1, \dots, i-1, \mbox{ but } d_1^{(i+1)}\nsucc b_2^{(i)} \mbox{ in } T \}, \mbox{ and }$$
$$\widetilde{\mathcal{T}}^{+,2(\infty)}:=\{\phi_1(T)\in \widetilde{\mathcal{T}}^{+,2}\,|\,d_1^{(j+1)}\succ b_2^{(j)} \mbox{ for all } j= 0, \dots,  m-2, \mbox{ in } T \} \, $$
be the images of the sets of tableaux satisfying the conditions $\langle 2$-$i\rangle$ and $\langle 2$-$\infty\rangle$ in Definition~\ref{def:sigma21}, respectively.

Then, we have the following lemmas whose proofs are given in Section~\ref{sec:proofs}.

\begin{lemma} \label{lem:disjoint from T2}
 The image $\widetilde{\sigma}_{3\rightarrow 2}^\square( \widetilde{\mathcal{T}}(m+k+1, m-2, 1)^- )$   of $\widetilde{\sigma}_{3\rightarrow 2}^\square$ is disjoint from $\widetilde{\mathcal T}^{+,2}$.
\end{lemma}

 \begin{lemma}\label{lem:disjoint from T1}
 The image of $\widetilde{\sigma}_{3\rightarrow 2}^\square$ restricted to the case (3) in Definition~\ref{def:sigma32} is disjoint from $\widetilde{\mathcal T}^{+,1}$.
 \end{lemma}

 \begin{definition}\label{def:phi2}
 We define $\phi_2: \widetilde{\mathcal{T}}(m+k+1, m-2, 1)^- \rightarrow \widetilde{\mathcal{T}}(m+k+1, m-1, 0)^+ $ for two integers $m\geq 3$ and $k\geq 0$ as follows:
Let
$$S=\scriptsize{ \ttableau{a_1& b_1^{(m-3)}& \cdots&b_1^{(2)}&b_1^{(1)}& b_1&c_1& d_1^{(k-1)}& \cdots&d_1^{(2)}& d_1^{(1)}&d_1& e_1\\
                                            a_2& b_2^{(m-3)}& \cdots&b_2^{(2)}& b_2^{(1)}\\
                                            a_3}}$$
be an $f$-tableau in $\widetilde{\mathcal{T}}(m+k+1, m-2, 1)^- \,.$

\begin{enumerate}
\item [(1)] When $b_1 \prec a_3$, set
         $$R_0:=\widetilde{\sigma}_{3\rightarrow 2}^\square(S)=
         \scriptsize{ \ttableau{a_1& b_1^{(m-3)}& \cdots&b_1^{(2)}&b_1^{(1)}& b_1&c_1& d_1^{(k-1)}& \cdots&d_1^{(2)}& d_1^{(1)}&d_1& e_1\\
                                            a_2& b_2^{(m-3)}& \cdots&b_2^{(2)}& b_2^{(1)}&\red{a_3}
                                            }}\,.$$

\begin{enumerate}
\item [$\bullet$] If $R_0 \not \in \widetilde {\mathcal T}^{+,1}$, then we let $\phi_2(S)=R_0$.
\item [$\bullet$] If $R_0 \in \widetilde{\mathcal T}^{+,1}$, then set
$$R_1:=  \scriptsize{ \ttableau{a_1& b_1^{(m-3)}& \cdots&b_1^{(2)}&b_1^{(1)}& b_1&c_1& d_1^{(k-1)}& \cdots&d_1^{(2)}& d_1^{(1)}&d_1& \red{a_2}\\
                                            \red{e_1}& b_2^{(m-3)}& \cdots&b_2^{(2)}& b_2^{(1)}& a_3
                                            }}\,.$$

\begin{enumerate}
\item [--] If $R_1 \not\in \widetilde{\mathcal T}^{+,2}$, then we let $\phi_2(S) =R_1$.


\item [--] If $R_1 \in \widetilde{\mathcal{T}}^{+,2 (i) }$ for some $0 \leq i <m-2$, then we let $\phi_2(S)$ be, where $b_2^{(0)}=a_3, d_1^{(0)}=d_1$,
  $$\hspace{2.8cm}\scriptsize{\ttableau{a_1&\cdots&b_1^{(i+1)}&b_1^{(i)}&\cdots &b_1^{(1)}&b_1& c_1& \cdots&d_1^{(i+1)}& \red{b_2^{(i)}}&\red{\cdots}& \red{b_2^{(1)}}&\red{a_3}& e_1\\
                                             a_2&\cdots& b_2^{(i+1)}&\red{d_1^{(i)}}&\red{\cdots} &\red{d_1^{(1)}}  &  \red{d_1} }}\,. $$

\item [--] If $R_1 \in \widetilde{\mathcal{T}}^{+,2 (m-2) }$, then we let $\phi_2(S)$ be
   $$\hspace{2.8cm}\scriptsize{\ttableau{a_1&b_1^{(m-3)}&\cdots  &b_1^{(1)}&b_1& c_1& \cdots&d_1^{(m)}& d_1^{(m-1)}   & \blue{e_1} &\red{b_2^{(m-3)}}&\red{\cdots}& \red{b_2^{(1)}}&\red{a_3}& \blue{d_1^{(m-2)}}\\
                                             a_2 & \red{d_1^{(m-3)}}& \red{ \cdots}  &\red{d_1^{(1)}}  &  d_1  }}\,. $$

\item [--] If $R_1 \in \widetilde{\mathcal{T}}^{+,2(\infty)}$,
 then we let $\phi_2(S)$ be
  $$\hspace{2.8cm}\scriptsize{\ttableau{a_1&b_1^{(m-3)}&\cdots  &b_1^{(1)}&b_1& c_1& \cdots&d_1^{(m)}& \blue{a_3} & \blue{e_1} &\red{b_2^{(m-3)}}&\red{\cdots}& \red{b_2^{(1)}}&\blue{d_1^{(m-1)}}& \blue{d_1^{(m-2)}}\\
                                             a_2 & \red{d_1^{(m-3)}}&\red{\cdots}   &\red{d_1^{(1)}}  &  d_1   } }\,. $$

\end{enumerate}

\end{enumerate}

\medskip \medskip
\item [(2)] When $b_1 \not\prec a_3$ and $b_1 \not \succ b_2^{(m-3)}$, set
          $$Q_0:=\widetilde{\sigma}_{3\rightarrow 2}^\square(S)=
         \scriptsize{ \ttableau{a_1& b_1^{(m-3)}& \cdots&b_1^{(2)}&b_1^{(1)}&\red{a_2}& c_1&d_1^{(k-1)}& \cdots& d_1^{(2)}& d_1^{(1)}&d_1& e_1\\
                                            \red{b_1}&b_2^{(m-3)}&\cdots&b_2^{(2)} & b_2^{(1)}&  a_3 }}\,.$$

\begin{enumerate}
    \item[$\bullet$] If $Q_0 \not \in\widetilde{\mathcal T}^{+,1}$, then we let $\phi_2(S)=Q_0$.

    \item[$\bullet$] If $Q_0 \in   \widetilde{\mathcal T}^{+,1}$, then we let
  $$\phi_2(S)=  \scriptsize{ \ttableau{a_1& b_1^{(m-3)}& \cdots&b_1^{(2)}&b_1^{(1)}&{a_2}& c_1&d_1^{(k-1)}& \cdots& d_1^{(2)}& d_1^{(1)}&d_1& \red{b_1}\\
                                            \red{e_1}&b_2^{(m-3)}&\cdots&b_2^{(2)} & b_2^{(1)}&  a_3 }}\,.$$

\end{enumerate}
\item [(3)] When $b_1 \not\prec a_3$ and $b_1   \succ b_2^{(m-3)}$, then  we let $\phi_2(S)=\widetilde{\sigma}_{3\rightarrow 2}^\square(S)$.

\end{enumerate}
\end{definition}

A proof of the following proposition is given in Section~\ref{sec:proofs}.

\begin{proposition} \label{prop:phi2} The map $\phi_2 : \widetilde{\mathcal{T}}(m+k+1, m-2, 1)^- \rightarrow \widetilde{\mathcal{T}}(m+k+1, m-1, 0)^+ $ in Definition~\ref{def:phi2} is a well defined injective map and satisfies the following property;
 $$\phi_1(\widetilde{\mathcal{T}}(m+k, m, 0)^-)\cap\phi_2(\widetilde{\mathcal{T}}(m+k+1, m-2, 1)^-)  =\varnothing\,,$$
where $\phi_1$ is the injective map $\widetilde{\sigma}_{2\rightarrow 1}^\square : \widetilde{\mathcal{T}}(m+k, m, 0)^- \rightarrow \widetilde{\mathcal{T}}(m+k+1, m-1, 0)^+ $ defined in Definition \ref{def:sigma21}.
\end{proposition}

\medskip
Now, if we define the map
$$\phi:   \widetilde{\mathcal T}(m+k,m,0)^- \cup \widetilde{\mathcal T}(m+k+1,m-2,1)^- \rightarrow \widetilde{\mathcal T}(m+k+1,m-1,0)^+ $$
by $$\phi|_{\widetilde{\mathcal T}(m+k,m,0)^-}=\phi_1 \quad \mbox{ and } \quad \phi|_{\widetilde{\mathcal T}(m+k+1,m-2,1)^-}=\phi_2\,,$$ then $\phi$ is an injective map. This completes a proof of Theorem~\ref{thm:main} in \textbf{Case I}.
\end{proof}

We close this subsection(Case I) with an example that illustrates the definitions of maps $\phi_1$ and $\phi_2$.
\begin{example}

Let $f=(2, 3, 5, 6, 7, 8, 8, 8)$ be the Hessenberg function we considered in Example~\ref{ex:23567888}. We let $m=3$ and $k=2$ in Case I, whose diagrams are given in Figure~\ref{fig:CaseIex}.

\begin{figure}[ht]
\begin{center}
$
\begin{tikzcd}[row sep=huge, column sep=small]
\mathcal{T}(4,2,2)^- \arrow[rrdd,"\sigma_{3\rightarrow 1}^{\square\!\square} ", hook] & \mathcal{T}(4, 3, 1)^+ \arrow[rd,"\sigma_{3\rightarrow 1}^\square", hook] & \\
\mathcal{T}(5,1,2)^+=\varnothing \arrow[rd,"\sigma_{3\rightarrow 1}^\square", hook] &&\mathcal{T}(5, 3, 0)^-\arrow[d, dashed]\\
& \mathcal{T}(6, 1, 1)^-\arrow[r,dashed]&{\mathcal{T}(6, 2, 0)^+}
\end{tikzcd}
$ \quad
$
\begin{tikzcd}[row sep=huge, column sep=huge]
\qquad\qquad & \widetilde{\mathcal{T}}(5,3, 0)^- \arrow[d,"\phi_1=\widetilde{\sigma}_{2\rightarrow 1}^\square", hook]\\
 \widetilde{\mathcal{T}}(6, 1, 1)^- \arrow[r,"\phi_2={(\widetilde{\sigma}_{3\rightarrow 2}^\square)^*}", hook]&  \widetilde{\mathcal{T}}(6, 2, 0)^+
\end{tikzcd}
$
 \end{center}
  \caption{\label{fig:CaseIex} An example of Case I. }
\end{figure}
We let 
$T_1=\tableau{1&3&2&6&\red{5}\\4&8&\red{7}}\,$,
$T_2=\tableau{1&3&2&8&\red{7}\\5&\red{6}&\red{4}}\,$, and
$T_3=\tableau{3&2&1&\red{7}&\red{8}\\\red{6}&\red{4}&\red{5}}$ be $f$-tableaux in                                                                      $\widetilde{\mathcal{T}}(5,3, 0)^-$. Note that $T_i$, $i=1, 2, 3$, are not in $\mathcal{T}(4, 3, 1)^+$ because $4\not\prec 5$, $5\not\prec 7$, $6\not\prec 8$, respectively.

Then,
\begin{itemize}
\item $T_1$ is of type $\langle 1\rangle$ and  $\phi_1(T_1)=\tableau{1&3&2&6&\red{5}&\red{7}\\4&8}\,$ since $5\not\succ 7$.

\item $T_2$ is type $\langle 2$-$0 \rangle$ and $\phi_1(T_2)=\tableau{1&3&2&8&\red{6}&\red{4}\\5&\red{7}}\,$ since $4\prec 7$ but $6\not\prec 8$.

\item $T_3$ is type $\langle 2$-$1 \rangle$ and $\phi_1(T_3)=\tableau{3&2&1&\red{6}&\red{4}&\red{5}\\ \red{7}&\red{8}}\,$ since $5\prec 8$ and $4\prec 7$ but $6\not\prec 1$.
\end{itemize}
It is easy to check that $\phi_1(T_i)\not\in \sigma_{3\rightarrow 1}^{\square\!\square}(\mathcal{T}(4,2,2)^-)$ for $i=1, 2, 3$. Note that type  $\langle 2$-$\infty \rangle$ does not occur because $m-2=1=k-1$.

\medskip

Let 
$S_1=\tableau{1&3&2&5&8&6\\4\\7}\,$, $S_2=\tableau{1&3&2&6&5&7\\4\\8}\,$ and
$S_3=\tableau{1&5&3&2&6&8\\4\\7}\,$ be $f$-tableaux in                                                                      $\widetilde{\mathcal{T}}(6, 1, 1)^-$. Then

\begin{itemize}
\item $S_1$ is of type (1) and $\widetilde{\sigma}_{3\rightarrow 2}^\square(S_1)=\tableau{1&3&2&5&\blue{8}&6\\\blue{4}&\red{7}}=R_0\,$. We can check that $R_0\not\in \widetilde {\mathcal T}^{+,1}$ since $4\prec 8$; if $R_0=\phi_1(T)=\widetilde{\sigma}_{2\rightarrow 1}^\square(T)$ then $T$ must be in $\mathcal{T}(4, 3, 1)^+$. Hence $\phi_2(S_1)=R_0$.

\item $S_2$ is of type (1) and $\widetilde{\sigma}_{3\rightarrow 2}^\square(S_2)=\tableau{1&3&2&6&5&\blue{7} \\\blue{4}&\red{8}}=R_0\,$. Then $R_0=\phi_1(T_1)$ for $T_1\in \widetilde{\mathcal{T}}(5,3, 0)^-$ given above. Hence we set $R_1:= \tableau{1&3&2&6&5&\blue{4} \\\blue{7}&8}\,$, and $R_1\not\in \widetilde {\mathcal T}^{+,2}$ because $3\not\prec 5$. Therefore $\phi_2(S_1)=R_1$.

\item $S_3$ is of type (2) and $\widetilde{\sigma}_{3\rightarrow 2}^\square(S_3)=\tableau{1&\red{4}&3&2&6&8\\\red{5}&\red{7}}=Q_0\,$. Then we can check that $Q_0\in\widetilde {\mathcal T}^{+,1}$; since $Q_0=\phi_1(T)$ for $T=\tableau{1&4&3&2&6\\5&7&8}\in \widetilde{\mathcal{T}}(5, 3, 0)^-$. Thus we let $\phi_2(S_3)=\tableau{1&4&3&2&6&\red{5}\\\red{8}&7}\,$.

\end{itemize}
\end{example}
\bigskip

\subsection{\textbf{Case II}}


Assume that $\mu_1=\mu_2+1$ and $\mu_2\not=\mu_3+1$.  Then from the diagram, Figure~\ref{fig:diagram II}, of $\mathcal{K}_\alpha\cup \mathcal{K}_\beta$, we need to injectively map the $f$-tableaux in $\mathcal{T}(\mu_1-1, \mu_2, \mu_3+1)\cup \mathcal{T}(\mu_1, \mu_2-1, \mu_3+1)$ into $\mathcal{T}(\mu_1-1, \mu_2-1, \mu_3+2)\cup \mathcal{T}(\mu_1, \mu_2, \mu_3)$.  Note that, if $\mu_2=\mu_3$ then $(\mu_1, \mu_2, \mu_3)$ is the only partition in the diagram of $\mathcal{K}_\alpha\cup \mathcal{K}_\beta$ and $c_\mu$ is nonnegative. Thus we assume that  $\mu_2>\mu_3+1$.

\begin{figure}[ht]
\begin{center}
$$
\begin{tikzcd}[row sep=huge]
{\mathcal{T}(\mu_1-1, \mu_2-1, \mu_3+2)^+}\arrow[rd,"\sigma_{3\rightarrow 1}^\square", hook]  & {\mathcal{T}(\mu_1-1, \mu_2, \mu_3+1)^- }\arrow[rd,"\sigma_{3\rightarrow 1}^\square", hook] &\\
&{\mathcal{T}(\mu_1, \mu_2-1, \mu_3+1)^-}\arrow[r,dashed]&{\mathcal{T}(\mu_1, \mu_2, \mu_3)^+}
\end{tikzcd}
$$
 \end{center}
  \caption{\label{fig:CaseII} Case II. }
\end{figure}

By Lemma~\ref{lem:3to1}, we have
$$\mathcal{T}(\mu_1-1, \mu_2, \mu_3+1)^-\simeq \sigma_{3\rightarrow 1}^\square(\mathcal{T}(\mu_1-1, \mu_2, \mu_3+1)^-)\subseteq \mathcal{T}(\mu_1, \mu_2, \mu_3)^+ \,,\,\, \mbox{  and} $$
$$\mathcal{T}(\mu_1-1, \mu_2-1, \mu_3+2)^+ \simeq \sigma_{3\rightarrow 1}^\square(\mathcal{T}(\mu_1-1, \mu_2-1, \mu_3+2)^+)\subseteq \mathcal{T}(\mu_1, \mu_2-1, \mu_3+1)^-\,,$$
as described in Figure ~\ref{fig:CaseII}.
Hence, we are left to define an injection $\phi$
\begin{align*}
&\mbox{from } \widetilde{\mathcal{T}}(\mu_1, \mu_2-1, \mu_3+1)^-=\mathcal{T}(\mu_1, \mu_2-1, \mu_3+1)^- - \sigma_{3\rightarrow 1}^\square(\mathcal{T}(\mu_1-1, \mu_2-1, \mu_3+2)^+) \\
&\mbox{to \quad} \widetilde{\mathcal{T}}(\mu_1, \mu_2, \mu_3)^+= \mathcal{T}(\mu_1, \mu_2, \mu_3)^+ -  \sigma_{3\rightarrow 1}^\square(\mathcal{T}(\mu_1-1, \mu_2, \mu_3+1)^-) \,.
\end{align*}
Note that, a tableau $T$ in $ \widetilde{\mathcal{T}}(\mu_1, \mu_2-1, \mu_3+1)^-$ can be written as $T=R\cup S$ where $R\in \mathcal{T}(\mu_3, \mu_3, \mu_3)$ and $S \in \mathcal{T}(\mu_1-\mu_3, \mu_2-\mu_3-1, 1)$. Since $\mu_2-1=\mu_1-2$, and we will manipulate only $S$ to define the image $\phi(T)$ of $T$, we may assume that $(\mu_1, \mu_2-1, \mu_3+1)=(m+2, m, 1)$ for $m\geq 1$.

\medskip
\begin{proof}[Proof of Theorem~\ref{thm:main} in \textbf{Case II}]

\quad
 We first give a definition of $\phi :  \widetilde{\mathcal{T}}(m+2, m, 1)^- \rightarrow \widetilde{\mathcal{T}}(m+2, m+1, 0)^+$, with the reasoning stated in parentheses that the given tableaux are contained in  $\widetilde{\mathcal{T}}(m+2, m+1, 0)^+$.
 $$ \mbox{Let  }S=\tableau{a_1&b_1&\cdots& c_1& d_1& e_1\\
                          a_2& b_2&\cdots& c_2\\
                          a_3} \mbox{\,\, be an element of \,}  \widetilde{\mathcal{T}}(m+2, m, 1)^-\,.$$
Then,  note that $b_2 \not\prec e_1$ since $S\not\in\sigma_{3\rightarrow 1}^\square(\mathcal{T}(m+1, m, 2))$.
\begin{enumerate}
\item[(1)] When $d_1\prec a_3$,
   \begin{itemize}
   \item[(1-1)] if $a_2\nprec e_1$, then let
              $$\phi(S)=  \tableau{a_1& b_1&\cdots& c_1& d_1& e_1\\
                                        a_2& b_2&\cdots& c_2& \red{a_3}}\,, \mbox{ and }$$

   \item[(1-2)] if  $a_2\prec e_1$, then let
             $$ \phi(S)=  \tableau{a_1& b_1&\cdots& c_1& d_1& \red{a_2}\\
                                 \red{e_1}& b_2&\cdots& c_2& \red{a_3}}\,. $$
                                 (Since $d_1 \in P_1 \cup P_2$ and $a_2 \in P_2$, we have $d_1 \not \succ a_2$.)
   \end{itemize}

\item[(2)] When $d_1\nprec a_3$ and $d_1\nsucc b_2$,
     (In this case we have $a_1 \prec d_1$ because $a_1 \prec a_2 \prec a_3$ and $d_1 \not\prec a_3$.)
   \begin{itemize}
   \item[(2-1)] if $d_1\nprec e_1$, then let
              $$\phi(S)=\tableau{a_1& b_1&\cdots& c_1& \red{a_2}& e_1\\
                                        \red{d_1}& b_2&\cdots& c_2& \red{a_3}}\,, \mbox{ and } $$
    \item[(2-2)] if  $d_1\prec e_1$, then let
             $$ \phi(S)=  \tableau{a_1& b_1&\cdots& c_1& \red{a_2}& \red{d_1}\\
                                             \red{e_1}& b_2&\cdots& c_2& \red{a_3}}\,. $$
             (Since $d_1 \not\prec a_3$ and $d_1 \prec e_1$, $d_1 \in P_2  $ and thus $a_2 \not \succ d_1$.)
   \end{itemize}

 \item[(3)] When $d_1\nprec a_3$ and $d_1\succ b_2$,
   \begin{itemize}
   \item[(3-1)] if $a_1\prec e_1$, then let
              $$\phi(S)=\tableau{a_1& b_1&\cdots& c_1& \red{a_2}& \red{d_1}\\
                                        \red{e_1}& b_2&\cdots& c_2& \red{a_3}}\,, \mbox{ and }$$

   \item[(3-2)] if  $a_1\nprec e_1$, then let
             $$ \phi(S)=  \tableau{a_1& b_1&\cdots& c_1& \red{e_1}& \red{d_1}\\
                                             a_2& b_2&\cdots& c_2& \red{a_3}}\,. $$
   (Since $a_1 \prec a_2 \prec a_3$ and $a_1 \not \prec e_1$, we have $e_1 \prec a_3$. Since $b_1 \prec b_2 \prec d_1$, we have $d_1 \in P_3$.  Since $e_1 \prec a_3$ and $d_1 \not\succ e_1$, we have $e_1 \in P_2$. Thus $c_1 \not \succ e_1$ and $e_1 \not \succ d_1$.)
   \end{itemize}
 \end{enumerate}

 We defined a map  $\phi :  \widetilde{\mathcal{T}}(m+2, m, 1)^- \rightarrow \widetilde{\mathcal{T}}(m+2, m+1, 0)^+$ and we now check that $\phi$ is injective:
 In each case (1-1) through (3-2), it is clear that $\phi$ is injective. We hence show that the image sets of $\phi$ for different cases are disjoint.

The image of $S$ of type (1) or of type (2) (respectively, of type (3)) is contained in the set of $f$-tableaux
$$  \tableau{\a_1& \b_1&\cdots& \c_1& \d_1& \e_1\\
                                        \a_2& \b_2&\cdots& \c_2& \d_2} $$
such that $\b_2 \not \prec \e_1$ (respectively, $\b_2 \prec \e_1$).

The image of $S$ of type (3-1) (respectively, of type (3-2)) is contained in the set of $f$-tableaux
$$  \tableau{\a_1& \b_1&\cdots& \c_1& \d_1& \e_1\\
                                        \a_2& \b_2&\cdots& \c_2& \d_2} $$
such that $\a_1 \prec \d_1$ (respectively, $\a_1 \not \prec \d_1$).

The image of $S$ of type (1-1)   (respectively, of type (1-2) or of type (2)) is contained in the set of $f$-tableaux
$$  \tableau{\a_1& \b_1&\cdots& \c_1& \d_1& \e_1\\
                                        \a_2& \b_2&\cdots& \c_2& \d_2} $$
such that $\a_2 \prec \d_2$ (respectively, $\a_2 \not \prec \d_2$).

The image of $S$ of type (1-2)  or of type (2-2)  (respectively,   of type (2-1)) is contained in the set of $f$-tableaux
$$  \tableau{\a_1& \b_1&\cdots& \c_1& \d_1& \e_1\\
                                        \a_2& \b_2&\cdots& \c_2& \d_2} $$
such that $\a_2 \succ \e_1$ (respectively, $\a_2 \not\succ \e_1$).

The image of $S$ of type (1-2)   (respectively,   of type (2-2)) is contained in the set of $f$-tableaux
$$  \tableau{\a_1& \b_1&\cdots& \c_1& \d_1& \e_1\\
                                        \a_2& \b_2&\cdots& \c_2& \d_2} $$
such that $\d_2 \succ \e_1$ (respectively, $\d_2 \not\succ \e_1$).

 This completes our proof of Theorem~\ref{thm:main} in Case II.
\end{proof}

\bigskip

\subsection{\textbf{Case III}}

Assume that $\mu_1\not=\mu_2+1$ and $\mu_2=\mu_3+1$.
Then, as in \textbf{Case II} it is enough to define an injective map $\phi$
\begin{align*}
&\mbox{from } \widetilde{\mathcal{T}}(\mu_1-1, \mu_2+1, \mu_3 )^-=\mathcal{T}(\mu_1-1, \mu_2+1, \mu_3)^- - \sigma_{3\rightarrow 1}^\square(\mathcal{T}(\mu_1-2, \mu_2+1, \mu_3+1)^+) \\
&\mbox{to \quad} \widetilde{\mathcal{T}}(\mu_1, \mu_2, \mu_3)^+= \mathcal{T}(\mu_1, \mu_2, \mu_3)^+ -  \sigma_{3\rightarrow 1}^\square(\mathcal{T}(\mu_1-1, \mu_2, \mu_3+1)^-)\,,
\end{align*}
as one can see in Figure ~\ref{fig:CaseIII}. Also note that we may assume $(\mu_1-1, \mu_2+1, \mu_3 )=(2+k,2,0)$ for $k\geq 1$.
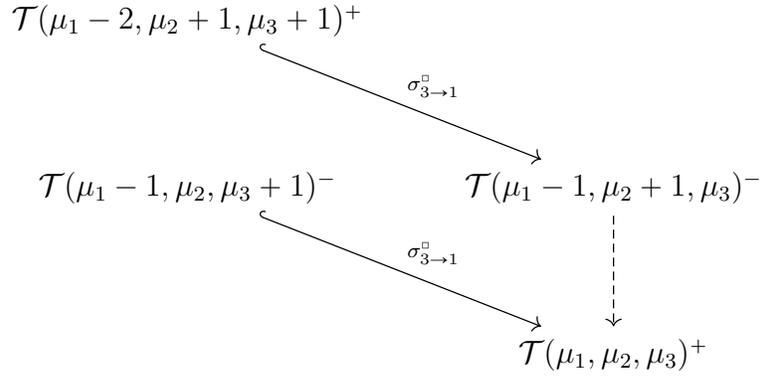
\begin{figure}[ht]
\begin{center}
$$
\begin{tikzcd}[row sep=huge]
 \mathcal{T}(\mu_1-2, \mu_2+1, \mu_3+1)^+ \arrow[rd,"\sigma_{3\rightarrow 1}^\square", hook]  & \\
 \mathcal{T}(\mu_1-1, \mu_2, \mu_3+1)^-
 \arrow[rd,"\sigma_{3\rightarrow 1}^\square", hook]
  &{\mathcal{T}(\mu_1-1, \mu_2+1, \mu_3)^-}\arrow[d,dashed]\\
&{\mathcal{T}(\mu_1, \mu_2, \mu_3)^+}
\end{tikzcd}
$$
 \end{center}
  \caption{\label{fig:CaseIII} Case III. }
\end{figure}

\medskip
\begin{proof}[Proof of Theorem~\ref{thm:main} in \textbf{Case III}]
\quad  We first give a definition of $\phi :  \widetilde{\mathcal{T}}(2+k, 2, 0)^- \rightarrow \widetilde{\mathcal{T}}(3+k, 1, 0)^+$, with the reasoning stated in parentheses that the given tableaux are contained in  $ \widetilde{\mathcal{T}}(3+k, 1, 0)^+$.
 $$ \mbox{Let  }T=\tableau{a_1&b_1&c_1&\cdots& d_1 \\
                          a_2& b_2 } \mbox{\,\, be an element of \,}  \widetilde{\mathcal{T}}(2+k , 2, 0)^-\,.$$
That is, $T$ is an $f$-tableau and  $a_2 \not\prec d_1$.

\begin{enumerate}
\item[$\langle 1\rangle$] When $d_1\not\succ b_2$,
   \begin{itemize}
   \item[$\langle 1$-$1\rangle$] if $a_2\not\prec b_2$, then let
              $$\phi(T)=  \tableau{a_1& b_1&c_1&\cdots&   d_1&\red{b_2}\\
                                        a_2 }\,, \mbox{ and } $$
              (Since $d_1 \not \succ b_2$ and $a_2 \not \prec b_2$, we have  $\phi(T) \in \widetilde{\mathcal T}(3+k, 1,0)$.)
   \item[$\langle 1$-$2\rangle$] if  $a_2\prec b_2$, then let
             $$ \phi(T)=  \tableau{a_1& b_1&c_1&\cdots&   d_1&\red{a_2}\\
                                        \red{b_2}  }\,. $$
             (Since $a_1 \prec a_2 \prec b_2$ and $a_2 \not \prec d_1$,   $\phi(T)$ is an $f$-tableau.)
   \end{itemize}

\item [$\langle 2\rangle$] When $d_1 \succ b_2$, (then, since $a_1 \not \succ b_1$ and $b_2 \succ b_1$, we have $a_1 <b_2$ and thus $a_1 \prec d_1$)
   \begin{itemize}
   \item[$\langle 2$-$1\rangle$] if $a_1 \not\prec b_2$ and $a_1 \not\succ c_1$, then  let
              $$\phi(T)=\tableau{\red{b_1}& \red{a_1}&c_1&\cdots&   d_1& a_2 \\
                                         b_2  } \,, $$
 (Since $b_1 \prec b_2$, we have $b_1 <a_1$.)

   \item[$\langle 2$-$2\rangle$] if  $a_1 \not\prec b_2$ and $a_1  \succ c_1$, then let
             $$ \phi(T)=  \tableau{a_1& b_1&c_1&\cdots&   \red{a_2}& b_2 \\
                                         \red{d_1} } \,, \mbox{ and }$$
             (Since $c_1 \prec a_1 \prec a_2$, we have $a_2 \in P_3$)
   \item[$\langle 2$-$3\rangle$] if  $a_1  \prec b_2$, then let
             $$ \phi(T)=  \tableau{a_1& b_1&c_1&\cdots&   d_1& \red{a_2} \\
                                         \red{b_2} }\,. $$
   \end{itemize}

 \end{enumerate}
Then $\phi$ is a map  from $\widetilde{\mathcal T}(2+k, 2,0)$ to $\widetilde{\mathcal T}(3+k, 1,0)$.
It remains to show that $\phi$ is injective.

The image of $T$ of type $\langle  1\rangle$ or of type $\langle 2$-$2 \rangle$ (respectively, of type $\langle 2$-$1 \rangle$ or of type $\langle 2$-$3 \rangle$) is contained in the set of $f$-tableaux
$$  \tableau{\a_1& \b_1&\c_1&\cdots&   \d_1&\e_1\\
                                        \a_2 } $$
such that $\a_2 \not \prec \d_1$ (respectively, $\a_2 \prec \d_1$). (If $T$ is of type $\langle 2$-$2 \rangle$, then both $\d_1$ and $\a_2$ are elements of $P_3$ and thus $\d_1 \not\prec \a_2$)

The image of $T$ of type $\langle 1$-$1 \rangle$ (respectively, of type $\langle 1$-$2 \rangle$ or of type $\langle 2$-$2 \rangle$) is contained in the set of $f$-tableaux
$$  \tableau{\a_1& \b_1&\c_1&\cdots&   \d_1&\e_1\\
                                        \a_2 } $$
such that $\a_2 \not \succ \e_1$ (respectively, $\a_2  \succ \e_1$).

The image of $T$ of type $\langle 1$-$2 \rangle$ (respectively,  of type $\langle 2$-$2 \rangle$) is contained in the set of $f$-tableaux
$$  \tableau{\a_1& \b_1&\c_1&\cdots&   \d_1&\e_1\\
                                        \a_2 } $$
such that $\a_1 \prec \e_1$ (respectively, $\a_1 \not \prec \e_1$).

The image of $T$ of type $\langle 2$-$1 \rangle$ (respectively,  of type $\langle 2$-$3 \rangle$) is contained in the set of $f$-tableaux
$$  \tableau{\a_1& \b_1&\c_1&\cdots&   \d_1&\e_1\\
                                        \a_2 } $$
such that $\a_2 \not \succ \b_1$ (respectively, $\a_2 \succ \b_1$).

Therefore, $\phi$ is an injective map from $\widetilde{\mathcal T}(2+k, 2,0)$ to $\widetilde{\mathcal T}(3+k, 1,0)$.
\end{proof}

\bigskip

\subsection{\bf Case IV} Assume $\mu_1=\mu_2+1$ and $\mu_2=\mu_3+1$. Then we have the diagram given in Figure~\ref{fig:CaseIV}, where $2 \mathcal{T}(\mu_1-1, \mu_2, \mu_3+1)^-$ means two copies of the set $\mathcal{T}(\mu_1-1, \mu_2, \mu_3+1)^-$.

\begin{figure}[ht]
\begin{center}
$$\begin{tikzcd}[row sep=huge]
 2 \mathcal{T}(\mu_1-1, \mu_2, \mu_3+1)^-\arrow[rd, hook] & \\
    &\mathcal{T}(\mu_1, \mu_2, \mu_3)^+
\end{tikzcd}
$$
 \end{center}
  \caption{\label{fig:CaseIV} Case IV. }
\end{figure}

\begin{proof}[Proof of Theorem~\ref{thm:main} in \textbf{Case IV}]

\quad Since $2(\mu_1-1, \mu_2, \mu_3+1)$ have negative sign and $(\mu_1, \mu_2, \mu_3)$ have positive sign in the diagram of $ \cup_{\alpha\in \mathcal{C}_\mu} \mathcal{K}_\alpha$, we need to show that $$2d(\mu_1-1, \mu_2, \mu_3+1)\leq d(\mu_1, \mu_2, \mu_3)\,,$$ where $d(\la)$ is the number of $f$-tableaux of shape $\la$.
We define two injective maps from $\mathcal{T}(\mu_1-1, \mu_2, \mu_3+1)$ to $\mathcal{T}(\mu_1, \mu_2, \mu_3)$ so that the image sets are disjoint.
Note that $\mu_1-1= \mu_2= \mu_3+1$. Hence if $T$ is an $f$-tableau in  $\mathcal{T}(\mu_1-1, \mu_2, \mu_3+1)$ then the last column of $T$ must have length $3$, say $\tableau{c_1\\c_2\\c_3}$ such that $c_1\prec c_2\prec c_3$.
\medskip

We let $\phi_1(T)$, $\phi_2(T)$ be the tableaux of shape $(\mu_1, \mu_2, \mu_3)$ such that the first $\mu_2-1$ columns are the same as the ones of $T$ and the last two columns are $\tableau{c_1&c_2\\c_3}$ and $\tableau{c_1&c_3\\c_2}$\,, respectively.
\medskip

Then it is easy to check that $\phi_1(T)$, $\phi_2(T)$ are $f$-tableaux in $\mathcal{T}(\mu_1, \mu_2, \mu_3)$ and $\phi_1$ and $\phi_2$ are injective with disjoint image sets.
\end{proof}

\bigskip


\section{Proofs of Lemmas and a Proposition in Section 4.1} \label{sec:proofs}

In this section, we prove Lemma \ref{lem:sigma21}, Lemma \ref{lem:sigma32}, Lemma \ref{lem:disjoint from T2}, Lemma  \ref{lem:disjoint from T1} and Proposition \ref{prop:phi2}, stated in Section~\ref{sec:h-positivity}.

\medskip
\begin{proof}[{\bf{Proof of Lemma \ref{lem:sigma21}}}]
Let

$$T=\scriptsize{\ttableau{a_1& b_1^{(m-3)} & \cdots &b_1^{(2)}& b_1^{(1)} &b_1& c_1& d_1^{(k-1)}&\cdots& d_1^{(2)}&d_1^{(1)}&d_1\\
                                            a_2& b_2^{(m-3)} &  \cdots &b_2^{(2)} & b_2^{(1)} & b_2& c_2}}$$
                                           be an $f$-tableau in $ \widetilde{\mathcal{T}}(m+k, m, 0)^- \,.$
                                           Here, $a_1=b_1^{(m-2)}$, $a_2=b_2^{(m-2)}$ and $c_1=d_1^{(k)}$.
Then $a_2 \not \prec d_1$ since $T\not\in \sigma_{3\rightarrow 1}^\square(\mathcal{T}(m+k-1, m, 1)^+)$.  \\

If $T$ is of type $\langle 1 \rangle$, then $\widetilde{\sigma}_{2\rightarrow 1}^\square(T)$ is  an $f$-tableau and, furthermore, is  an element of $\widetilde{\mathcal{T}}(m+k+1, m-1, 0)^+ $ because  $a_2 \not \prec d_1$.

If $T$ is of type $\langle 2$-$i\rangle$, then, since $d_1 \succ c_2 \succ c_1$ and $b_1 \not \succ c_1$, we have $d_1 \succ b_1$. Similarly, $d_1^{(j+1)} \succ b_1 ^{(j+1)}$ for all $j=-1, 0, ..., i-1$. From $d_1^{(i)} \in P_3$ and $c_1 \in P_1$, it follows that $\widetilde{\sigma}_{2\rightarrow 1}^\square(T)$ is an $f$-tableau. Since $c_2 \in P_2$, $\widetilde{\sigma}_{2\rightarrow 1}^\square(T)$ is an element of $\widetilde{\mathcal T}(m+k+1, m-1,0)^+$.

If $T$ is of type $\langle 2$-$\infty\rangle$, then, since $d_1 \not \succ a_2$ and $d_1^{(m-1)} \succ a_2$, we have $d_1^{(m-1)} > d_1$ and thus $b_1 \prec d_1^{(m-1)}$ and $c_2 \prec d_1^{(m-1)}$. Hence $\widetilde{\sigma}_{2\rightarrow 1}^\square(T)$ is an element of $\widetilde{\mathcal T}(m+k+1, m-1,0)^+$. \\

The image of $T$ of type $\langle  1\rangle$ (respectively, of type $\langle 2 \rangle$) is contained in the set of tableaux
$$
\scriptsize{\ttableau{\a_1& \b_1^{(m-3)} & \cdots &\b_1^{(2)}& \b_1^{(1)} &\b_1& \c_1& \d_1^{(k-1)}&\cdots& \d_1^{(2)}&\d_1^{(1)}&\d_1& \red{ \e_1}\\
                                            \a_2& \b_2^{(m-3)} &  \cdots &\b_2^{(2)} & \b_2^{(1)} & \red{\b_2}}}\, $$
such that $\b_2 \not \succ \e_1$ (respectively, $\b_2 \succ \e_1$).

If $T$ is of type $\langle 2$-$i\rangle$, i.e., there is $i\leq \min(k-1, m-2)$ such that $d_1^{(j+1)}\succ b_2^{(j)}$ for all $j=-1, 0, \dots, i-1$ but $d_1^{(i+1)}\nsucc b_2^{(i)}$,  then $\widetilde{\sigma}_{2\rightarrow 1}^\square(T)$ is contained in the set of tableaux
$$
\scriptsize{\ttableau{\a_1& \b_1^{(m-3)} & \cdots &\b_1^{(2)}& \b_1^{(1)} &\b_1& \c_1& \d_1^{(k-1)}&\cdots& \d_1^{(2)}&\d_1^{(1)}&\d_1& \e_1\\
                                            \a_2& \b_2^{(m-3)} &  \cdots &\b_2^{(2)} & \b_2^{(1)} & \b_2}}\,. $$
such that $\b_2^{(j+1)} \succ \d_1^{(j)}$ for all $j=0,..., i-1$ and $\b_2^{(i+1)} \not \succ \d_1^{(i)}$. Therefore, $\widetilde{\sigma}_{2\rightarrow 1}^\square$ restricted to the set of tableau $T$ of type $\langle 2$-$i \rangle$ is injective.

To show that the map $\widetilde{\sigma}_{2\rightarrow 1}^\square$ is injective, it suffices to show that the image of $\widetilde{\sigma}_{2\rightarrow 1}^\square$ restricted to the set of tableau of type $\langle 2$-$i \rangle$ with $i=(m-2)$ is disjoint of the image of $\widetilde{\sigma}_{2\rightarrow 1}^\square$ restricted to the set of tableau of type $\langle 2$-$\infty \rangle$.
The first is of the form

$$
\scriptsize{ \ttableau{\a_1&\cdots&\b_1^{(2)}&\b_1^{(1)}&\b_1& \c_1& \cdots &\d_1^{(m)}& \d_1^{(m-1)}  &\red{\b_2^{(m-2)}}&\red{ \cdots} &\red{\b_2^{(0)}}& \red{\c_2}\\
                                            \red{\d_1^{(m-2)}}&\red{\cdots} & \red{\d_1^{(2)}} &  \red{\d_1^{(1)}}&  \d_1 }}                                 \,  $$

                                            with $\d_1 \not \succ \b_2^{(m-2)} =\a_2$
and the second is of the form

$$
 \scriptsize{ \ttableau{\a_1&\cdots&\b_1^{(2)}&\b_1^{(1)}&\b_1& \c_1& \cdots &\d_1^{(m)}& \blue{\d_1}&\red{\b_2^{(m-2)}}&\red{ \cdots} &\red{\b_2^{(0)}}& \red{\c_2}\\
                                            \red{\d_1^{(m-2)}}&\red{\cdots} & \red{\d_1^{(2)}} &  \red{\d_1^{(1)}}& \blue{\d_1^{(m-1)}}}}                                 \,  $$

                                            with $\d_1 ^{(m-1)}  \succ \b_2^{(m-2)} =\a_2$.
   Consequently, $\widetilde{\sigma}_{2\rightarrow 1}^\square$ is injective.
\end{proof}

\medskip

\begin{proof}[{\bf{Proof of Lemma \ref{lem:sigma32}}}]
Let

 $$S=\scriptsize{ \ttableau{a_1& b_1^{(m-3)}& \cdots&b_1^{(2)}&b_1^{(1)}& b_1&c_1& d_1^{(k-1)}& \cdots&d_1^{(2)}& d_1^{(1)}&d_1& e_1\\
                                            a_2& b_2^{(m-3)}& \cdots&b_2^{(2)}& b_2^{(1)}\\
                                            a_3}}$$
be an $f$-tableau in $\widetilde{\mathcal{T}}(m+k+1, m-2, 1) \,.$  Here,  $a_1=b_1^{(m-2)} $ and $a_2=b_2^{(m-2)}  $ and $c_1=d_1^{(k)}$. Then, since $S\not\in \sigma_{3\rightarrow 1}^\square(\mathcal{T}(m+k-2, m, 2)^+)$ we have $b_2^{(m-3)} \not\prec e_1$. \\

\noindent {\bf Claim.} If $b_1 \not \prec a_3$, then $a_1 \prec b_1$  and $a_2 \not \succ c_1$. \\

\noindent {\it Proof of} {\bf  Claim.}
If $b_1 \not \prec a_3$, then, from $ a_2 \prec a_3$ and $b_1 \not \prec a_3$, it follows that $a_2 <b_1$ and thus  $a_1 \prec b_1$ (This is essentially the $(3+1)$-free condition). Since $a_2 <b_1$ and $b_1^{(1)} \not\succ b_1$, we have $a_2\not \succ c_1$. \qed \\

If $S$ is of type  (1), then $\widetilde{\sigma}_{3\rightarrow 2}^\square(S)$ is     an element of $\widetilde{\mathcal{T}}(m+k+1, m-1, 0)^+ $.

If  $S$ is of type  (2),   then by {\bf Claim.} together with the property $a_2 \in P_2$, $\widetilde{\sigma}_{3\rightarrow 2}^\square(S)$ is an element of $\widetilde{\mathcal{T}}(m+k+1, m-1, 0)^+ $.

If $S$ is of type (3), i.e., $b_1 \not \prec a_3$ and $b_1 \succ b_2^{(m-3)}$, then $b_1 \in P_3$ and thus $c_1 \in P_2 \cup P_3$. \\

Assume that $S$ is of type (3) and  $e_1 \not \prec a_3$ or $a_2 \prec d_1$. Then $a_1 \prec e_1$. To see this, use $a_1 \prec a_2 \prec a_3$ and $e_1 \not \prec a_3$ or $a_1 \prec a_2 \prec d_1$ and $e_1 \not \prec d_1$.  Thus, if $S$ is of type (3-1), then $\widetilde{\sigma}_{3\rightarrow 2}^\square(S)$ is an element of $\widetilde{\mathcal{T}}(m+k+1, m-1, 0)^+ $.\\

Now assume that  $S$ is of type (3) and $e_1 \prec a_3$. If $b_1^{(1)} \not \succ e_1$, then
$$         \scriptsize{ \ttableau{a_1& b_1^{(m-3)}& \cdots&b_1^{(2)}&b_1^{(1)}& \red{e_1}&c_1& d_1^{(k-1)}& \cdots&d_1^{(2)}& d_1^{(1)}&d_1& \red{b_1}\\
                                            a_2& b_2^{(m-3)}& \cdots&b_2^{(2)}& b_2^{(1)}& a_3
                                            }}\, $$
is an element of $\widetilde{\mathcal{T}}(m+k+1, m-1, 0)^+ $. Here, $e_1 \not \succ c_1$ follows from the property that $e_1 \in P_1 \cup P_2$ and  $c_1 \in P_2 \cup P_3$.

 If $b_1^{(1)} \succ e_1=d_1^{(-1)}$ and $b_1^{(2)}  \succ d_1=d_1^{(0)} $, then, from $b_2^{(1)} \succ b_1^{(1)} \succ e_1$ and $d_1 \not \succ e_1$, it follows that $b_1^{(1)} > d_1$ and thus $d_1 \prec b_2^{(1)}$ and $d_1^{(1)} \not \succ b_1^{(1)}$. Therefore,
$$         \scriptsize{ \ttableau{a_1& b_1^{(m-3)}& \cdots&b_1^{(2)}&\red{d_1} & \red{e_1}&c_1& d_1^{(k-1)}& \cdots&d_1^{(2)}& d_1^{(1)}&\red{b_1^{(1)}}& \red{b_1}\\
                                            a_2& b_2^{(m-3)}& \cdots&b_2^{(2)}& b_2^{(1)}& a_3
                                            }}\, $$
is an element of $\widetilde{\mathcal{T}}(m+k+1, m-1, 0)^+ $.
Similarly, if $b_1^{(1)} \succ e_1=d_1^{(-1)}, b_1^{(2)}  \succ d_1^{(0)}=d_1, ...,  b_1^{(i+1)}    \succ d_1^{(i-1)}, b_1^{(i+2)}  \not \succ d_1^{(i)}$, then
$$          \scriptsize{ \ttableau{a_1& \cdots & b_1^{(i+2)}&\red{d_1^{(i)}}&\red{\cdots}&\red{d_1^{(0)}}&\red{e_1}& c_1&\cdots &d_1^{(i+1)}&\red{b_1^{(i+1)}}&\cdots& \red{b_1^{(1)}}&\red{b_1}\\
                                           a_2&\cdots& b_2^{(i+2)}&b_2^{(i+1)} &\cdots& b_2^{(1)}& a_3 }}\, $$
is an element of $\widetilde{\mathcal{T}}(m+k+1, m-1, 0)^+ $. Thus, if $S$ is of type (3-2), then $\widetilde{\sigma}_{3\rightarrow 2}^\square(S)$ is an element of $\widetilde{\mathcal{T}}(m+k+1, m-1, 0)^+ $.\\

The image of $S$ of type (1) or of type (2) (respectively, of type (3)) is contained in the set of tableaux
$$
\scriptsize{\ttableau{\a_1& \b_1^{(m-3)} & \cdots &\b_1^{(2)}& \b_1^{(1)} &\b_1& \c_1& \d_1^{(k-1)}&\cdots& \d_1^{(2)}&\d_1^{(1)}&\d_1& \red{\e_1}\\
                                            \a_2& \red{\b_2^{(m-3)}} &  \cdots &\b_2^{(2)} & \b_2^{(1)} & \b_2}}\,. $$
such that $\b_2^{(m-3)} \not \prec \e_1$ (respectively, $\b_2^{(m-3)} \prec \e_1$).

The image of $S$ of type (1)  (respectively, of type (2)) is contained in the set of tableaux
$$
\scriptsize{\ttableau{\a_1& \b_1^{(m-3)} & \cdots &\b_1^{(2)}& \b_1^{(1)} &\b_1& \c_1& \d_1^{(k-1)}&\cdots& \d_1^{(2)}&\d_1^{(1)}&\d_1& \e_1\\
                                            \red{\a_2}& \b_2^{(m-3)} &  \cdots &\b_2^{(2)} & \b_2^{(1)} & \red{\b_2}}}\,. $$
such that $\a_2 \prec \b_2$ (respectively, $\a_2 \not \prec \b_2$).

The proof for the injectivity of $\widetilde{\sigma}_{3\rightarrow 2}^\square$ restricted to the set of $f$-tableaux $S$ of type (3) is similar to the previous cases and we omit the proof.
\end{proof}

\medskip

\begin{proof}[{\bf{Proof of Lemma ~\ref{lem:disjoint from T2}}}]
 Any  $f$-tableaux
$$  \scriptsize{\ttableau{\a_1& \b_1^{(m-3)} & \cdots &\b_1^{(2)}& \b_1^{(1)} &\b_1& \c_1& \d_1^{(k-1)}&\cdots& \d_1^{(2)}&\d_1^{(1)}&\d_1&\e_1\\
                                            \a_2& \b_2^{(m-3)} &  \cdots &\b_2^{(2)} & \b_2^{(1)} & \b_2}}$$
 in    $\widetilde{\mathcal T}^{+,2(i)}$ satisfies $\a_2 \not\prec \b_2$, $\c_1 \prec \e_1$, $\b_2 \succ \e_1$ and $\b_1^{(j)} \prec \d_1^{(j)} \prec \b_2^{(j+1)}$ for $j=0,...,i-1,$ where $\d_1^{(0)}:=\d_1$. Moreover, $\b_1^{(i)} \prec \d_1^{(i)} \not \prec \b_2^{(i+1)}$ and $\d_1^{(i+1)} \not\succ \b_2^{(i)}$.

 Any $f$-tableaux
$$  \scriptsize{\ttableau{\a_1& \b_1^{(m-3)} & \cdots &\b_1^{(2)}& \b_1^{(1)} &\b_1& \c_1& \d_1^{(k-1)}&\cdots& \d_1^{(2)}&\d_1^{(1)}&\d_1&\e_1\\
                                            \a_2& \b_2^{(m-3)} &  \cdots &\b_2^{(2)} & \b_2^{(1)} & \b_2}}$$
in $\widetilde{\mathcal T}^{+,2(\infty)}$ satisfies  $\a_2 \not\prec \b_2$, $\c_1 \prec \e_1$ and  $\b_1^{(j)} \prec \d_1^{(j)} \prec \b_2^{(j+1)}$ for $j=0,...,m-2,$ where $\d_1^{(0)}:=\d_1$, and $\b_1 \prec \d_1^{(m-1)}$. Furthermore, $\b_2 \succ \e_1$ by Remark \ref{rmk:T 2 infty}.

 Therefore, the set  $\widetilde{\mathcal T}^{+,2} $ is contained in the set of $f$-tableaux
$$  \scriptsize{\ttableau{\a_1& \b_1^{(m-3)} & \cdots &\b_1^{(2)}& \b_1^{(1)} &\b_1& \c_1& \d_1^{(k-1)}&\cdots& \d_1^{(2)}&\d_1^{(1)}&\d_1&\e_1\\
                                            \a_2& \b_2^{(m-3)} &  \cdots &\b_2^{(2)} & \b_2^{(1)} & \b_2}}$$
 such that $\a_2 \not\prec \b_2$ and $\c_1 \prec \e_1$ and $\b_2 \succ \e_1$.

The image $\widetilde{\sigma}_{3\rightarrow 2}^\square(S)$ of $S$ of type (1) and (3-2) (respectively, of type (3-1)) in Definition~\ref{def:sigma32} is an $f$-tableau
$$  \scriptsize{\ttableau{\a_1& \b_1^{(m-3)} & \cdots &\b_1^{(2)}& \b_1^{(1)} &\b_1& \c_1& \d_1^{(k-1)}&\cdots& \d_1^{(2)}&\d_1^{(1)}&\d_1&\e_1\\
                                            \a_2& \b_2^{(m-3)} &  \cdots &\b_2^{(2)} & \b_2^{(1)} & \b_2}}$$
 satisfying $\a_2 \prec \b_2$ (respectively, $\c_1 \not \prec \e_1$). Thus they are not contained in  $\widetilde{\mathcal T}^{+,2}$.

 Suppose that the image
  $$\widetilde{\sigma}_{3\rightarrow 2}^\square(S):=
         \scriptsize{ \ttableau{a_1& b_1^{(m-3)}& \cdots&b_1^{(2)}&b_1^{(1)}&\red{a_2}& c_1&d_1^{(k-1)}& \cdots& d_1^{(2)}& d_1^{(1)}&d_1& e_1\\
                                            \red{b_1}&b_2^{(m-3)}&\cdots&b_2^{(2)} & b_2^{(1)}& \blue{ a_3} }}$$
 of $S$ of type (2) in Definition~\ref{def:sigma32} is contained in  $\widetilde{\mathcal T}^{+,2}$. Then $c_1 \prec e_1 \prec a_3$ and $b_1 \not \prec a_3$ must hold, which implies that $c_1 \prec b_1$ by (3+1)-free condition in Lemma~\ref{lem:basic}. This is a contradiction since $b_1\nsucc c_1$ in $S$. Therefore, the image $\widetilde{\sigma}_{3\rightarrow 2}^\square(S)$ is not contained in $\widetilde{\mathcal T}^{+,2}$.
 \end{proof}

\medskip

\begin{proof}[{\bf{Proof of Lemma~\ref{lem:disjoint from T1}}}] The set  $\widetilde{\mathcal T}^{+,1}$ consists of $f$-tableaux
$$  \scriptsize{\ttableau{\a_1& \b_1^{(m-3)} & \cdots &\b_1^{(2)}& \b_1^{(1)} &\b_1& \c_1& \d_1^{(k-1)}&\cdots& \d_1^{(2)}&\d_1^{(1)}&\d_1&\e_1\\
                                            \a_2& \b_2^{(m-3)} &  \cdots &\b_2^{(2)} & \b_2^{(1)} & \b_2}}$$
such that $\c_1 \prec \e_1$, $\b_2 \not\succ \e_1$ and $\a_2 \not\prec \d_1$; where the last condition is from the fact that the $f$-tableaux in $\widetilde{\mathcal T}^{+,1}$ are $\widetilde{\sigma}_{2\rightarrow 1}^\square(T)$ for $T\not \in \widetilde{\mathcal T}(m+k-1, m, 1)^+$.
 The image $\widetilde{\sigma}_{3\rightarrow 2}^\square(S)$ of $S$ of type (3)  is an $f$-tableau
$$  \scriptsize{\ttableau{\a_1& \b_1^{(m-3)} & \cdots &\b_1^{(2)}& \b_1^{(1)} &\b_1& \c_1& \d_1^{(k-1)}&\cdots& \d_1^{(2)}&\d_1^{(1)}&\d_1&\e_1\\
                                            \a_2& \b_2^{(m-3)} &  \cdots &\b_2^{(2)} & \b_2^{(1)} & \b_2}}$$
 satisfying $\c_1 \not \prec \e_1$.
 \end{proof}

\medskip

\begin{proof}[{\bf Proof of Proposition \ref{prop:phi2}}]
Let
$$S=\scriptsize{ \ttableau{a_1& b_1^{(m-3)}& \cdots&b_1^{(2)}&b_1^{(1)}& b_1&c_1& d_1^{(k-1)}& \cdots&d_1^{(2)}& d_1^{(1)}&d_1& e_1\\
                                            a_2& b_2^{(m-3)}& \cdots&b_2^{(2)}& b_2^{(1)}\\
                                            a_3}}$$
be an $f$-tableau in $\widetilde{\mathcal{T}}(m+k+1, m-2, 1) \,.$
Then $b_2^{(m-3)} \not\prec e_1$ and here, $a_1=b_1^{(m-2)}  $ and $a_2=b_2^{(m-2)}$.

\begin{enumerate}
\item [(1)] When $b_1 \prec a_3$, we set
         $$R_0:=\widetilde{\sigma}_{3\rightarrow 2}^\square(S)=
         \scriptsize{ \ttableau{a_1& b_1^{(m-3)}& \cdots&b_1^{(2)}&b_1^{(1)}& b_1&c_1& d_1^{(k-1)}& \cdots&d_1^{(2)}& d_1^{(1)}&d_1& e_1\\
                                            a_2& b_2^{(m-3)}& \cdots&b_2^{(2)}& b_2^{(1)}&\red{a_3}
                                            }}\,.$$
Then by Lemma \ref{lem:disjoint from T2}, $R_0 \not \in\widetilde {\mathcal T}^{+,2}$.
\begin{enumerate}
\item [$\bullet$] If $R_0 \not \in \widetilde {\mathcal T}^{+,1}$, then we let $\phi_2(S)=R_0$.
\item [$\bullet$] If $R_0 \in \widetilde{\mathcal T}^{+,1}$, i.e.  $c_1 \prec e_1$, $a_3 \not\succ e_1$ and $a_2 \not\prec d_1$, then we set
$$R_1:=  \scriptsize{ \ttableau{a_1& b_1^{(m-3)}& \cdots&b_1^{(2)}&b_1^{(1)}& b_1&c_1& d_1^{(k-1)}& \cdots&d_1^{(2)}& d_1^{(1)}&d_1& \red{a_2}\\
                                            \red{e_1}& b_2^{(m-3)}& \cdots&b_2^{(2)}& b_2^{(1)}& a_3
                                            }}\,.$$
Then  $R_1 $ is an $f$-tableau  (because $a_3 \not\succ e_1$, $a_2 \not \prec d_1$ and $b_2^{(m-3)} \not \prec e_1$), and  is not contained in $ \widetilde{\mathcal T}^{+,1}$ (because $a_3 \succ a_2$), and is not of the form $R_0$ (because $e_1 \not \prec a_3$).

\begin{enumerate}
\item [--] If $R_1 \not\in \widetilde{\mathcal T}^{+,2}$,

then we let $\phi_2(S) =R_1$.




\item [--] If $R_1 \in \widetilde{\mathcal{T}}^{+,2 (i) }$ for some $0\leq i <m-2$,

 then we have ($c_1 \prec a_2$, $b_1 \prec d_1$, $b_1^{(j)} \prec d_1^{(j)}$ for $j=0,\dots, i$, and $d_1^{(j)} \prec b_2^{(j+1)}$ for $j=-1,\dots, i-1$, where $d_1^{(0)} =d_1$). Moreover, we have ($b_2^{(i+1)} \not \succ d_1^{(i)}$ and $d_1^{(i+1)} \not \succ b_2^{(i)}$),
and we let $\phi_2(S)=:R_2$ be, where $b_2^{(-1)}=a_3$ and $d_1^{(0)}=d_1$,
  $$\hspace{3cm} \scriptsize{\ttableau{a_1&\cdots&b_1^{(i+1)}&b_1^{(i)}&\cdots &b_1^{(1)}&b_1& c_1& \cdots&d_1^{(i+1)}& \red{b_2^{(i)}}&\red{\cdots}& \red{b_2^{(1)}}&\red{a_3}& e_1\\
                                             a_2&\cdots& b_2^{(i+1)}&\red{d_1^{(i)}}&\red{\cdots} &\red{d_1^{(1)}}  &  \red{d_1} }}. $$
Then $R_2$ is an $f$-tableau (because $e_1 \not \prec a_3$), and is contained neither in $\widetilde{\mathcal T}^{+,1}$ nor in $\widetilde{\mathcal T}^{+,2}$ (because $a_2 \prec a_3$ and $d_1 \not \succ e_1$), and is not of the form $R_0$, $R_1$ (because $a_2 \not \prec d_1$, $d_1 \not \succ e_1$).

\item [--] If $R_1 \in \widetilde{\mathcal{T}}^{+,2 (m-2) }$,

then we have ($c_1 \prec a_2$, $b_1 \prec d_1$, $b_1^{(j)} \prec d_1^{(j)}$ for $j=1,\dots, m-2$, and $d_1^{(j)} \prec b_2^{(j+1)}$ for $j=0,\dots, m-4$, where $d_1^{(0)} =d_1$). Moreover, we have ($d_1^{(m-3)} \prec e_1$, $a_3 \not\succ d_1^{(m-2)}$ and $d_1^{(m-1)} \not \succ e_1$),
and we let $\phi_2(S)=:R_2$ be
   $$ \hspace{3cm} \scriptsize{\ttableau{a_1&b_1^{(m-3)}&\cdots  &b_1^{(1)}&b_1& c_1& \cdots&d_1^{(m)}& d_1^{(m-1)}   & \blue{e_1} &\red{b_2^{(m-3)}}&\red{\cdots}& \red{b_2^{(1)}}&\red{a_3}& \blue{d_1^{(m-2)}}\\
                                             a_2 & \red{d_1^{(m-3)}}& \red{ \cdots}  &\red{d_1^{(1)}}  &  d_1  }}. $$
Then $R_2$ is an $f$-tableau (because both $a_2$ and $d_1^{(m-3)}$ belong to $P_2$), and is contained neither in $\widetilde{\mathcal T}^{+,1}$ nor in $\widetilde{\mathcal T}^{+,2}$ (because $a_2 \prec a_3$ and $d_1 \not \succ d_1^{(m-2)}$ for $d_1\in P_2$ and $d_1^{(m-2)} \in P_2\cup P_3$), and is not of the form $R_0$, $R_1$ (because $e_1 \not \prec d_1$, $d_1 \not \succ d_1^{(m-2)}$).

\item [--] If $R_1 \in \widetilde{\mathcal{T}}^{+,2(\infty)}$,

 then we have ($c_1 \prec a_2$, $b_1 \prec d_1$, $b_1^{(j)} \prec d_1^{(j)}$ for $j=1,\dots, m-2$, and  $d_1^{(j)} \prec b_2^{(j+1)}$ for $j=0,\dots, m-4$, where $d_1^{(0)} =d_1$).  Moreover, we have ($d_1^{(m-3)} \prec e_1$, $b_1\prec d_1^{(m-1)}$, $a_3 \succ d_1^{(m-2)}$, $a_3 \not\prec d_1^{(m)}$ and $d_1^{(m-1)} \succ a_2$),

and we let $\phi_2(S)=:R_2$ be
  $$\hspace{3cm} \scriptsize{\ttableau{a_1&b_1^{(m-3)}&\cdots  &b_1^{(1)}&b_1& c_1& \cdots&d_1^{(m)}& \blue{a_3} & \blue{e_1} &\red{b_2^{(m-3)}}&\red{\cdots}& \red{b_2^{(1)}}&\blue{d_1^{(m-1)}}& \blue{d_1^{(m-2)}}\\
                                             a_2 & \red{d_1^{(m-3)}}&\red{\cdots}   &\red{d_1^{(1)}}  &  d_1   } }. $$

Then $R_2$ is an $f$-tableau (because  both $a_3$ and $d_1^{(m-1)}$ are elements of $P_3$), and is contained neither in $\widetilde{\mathcal T}^{+,1}$ nor in $\widetilde{\mathcal T}^{+,2}$ (because $a_2 \prec d_1^{(m-1)}$, and $d_1 \not \succ d_1^{(m-2)}$ for $d_1, d_1^{(m-2)} \in P_2$), and is not of the form $R_0$, $R_1$ (because $a_2 \not \prec d_1 $ and $d_1 \not \succ d_1^{(m-2)}$ for both $d_1$ and $d_1^{(m-2)}$ are contained in $P_2$).
\end{enumerate}

\medskip
Furthermore, $R_2$'s are all different.
For,
if $R_1 \in \widetilde{\mathcal T}^{+, 2(i)}$ for some $0 \leq i \leq m-3$, then $R_2$ is contained in the set of $f$-tableaux
$$\hspace{2cm} \scriptsize{\ttableau{\a_1&\cdots&\b_1^{(i+1)}&\b_1^{(i)}&\cdots &\b_1^{(1)}&\b_1& \c_1& \cdots&\d_1^{(i+1)}&  \d_1^{(i)} & \cdots & \d_1^{(1)} &\d_1& \e_1\\
                                             \a_2&\cdots& \b_2^{(i+1)}&\b_2^{(i)} & \cdots  & \b_2^{(1)}   &  \a_2 }}  $$
such that $\b_1^{(j)} \prec \d_1^{(j)}$ for all $j=0,...,i$, where $\d_1^{(0)}=\d_1$, and $\b_2^{(j)} \prec \d_1^{(j+1)}$ for $j=1,\dots ,i-1$, and $\b_2^{(i)}\not \prec \d_1^{(i+1)}$.

On the other hand, if $R_1 \in \widetilde{\mathcal T}^{+, 2(m-2)}$ (respectively, $\widetilde{\mathcal T}^{+, 2(\infty)}$), then $R_2$  is contained in the set of $f$-tableaux
$$\hspace{2cm} \scriptsize{\ttableau{\a_1&\cdots&\b_1^{(i+1)}&\b_1^{(i)}&\cdots &\b_1^{(1)}&\b_1& \c_1& \cdots&\d_1^{(i+1)}&  \d_1^{(i)} & \cdots & \d_1^{(1)} &\d_1& \e_1\\
                                             \a_2&\cdots& \b_2^{(i+1)}&\b_2^{(i)} & \cdots  & \b_2^{(1)}   &  \a_2 }}  $$
such that $\b_1^{(j)} \prec \d_1^{(j)}$ for all $j=0,\dots , m-2$, where $\d_1^{(0)}=d_1$, and $\b_2^{(j)} \prec \d_1^{(j+1)}$ for $j=1,\dots , m-4$ and $\b_2^{(m-3)}  \prec \d_1^{(m-2)}$, and $\d_1^{(m-1)} \not \succ \e_1$ (respectively, $\d_1^{(m-1)}  \succ \e_1$).

\end{enumerate}

\medskip \medskip
\item [(2)] When $b_1 \not\prec a_3$ and $b_1 \not \succ b_2^{(m-3)}$, we set
          $$Q_0:=\widetilde{\sigma}_{3\rightarrow 2}^\square(S)=
         \scriptsize{ \ttableau{a_1& b_1^{(m-3)}& \cdots&b_1^{(2)}&b_1^{(1)}&\red{a_2}& c_1&d_1^{(k-1)}& \cdots& d_1^{(2)}& d_1^{(1)}&d_1& e_1\\
                                            \red{b_1}&b_2^{(m-3)}&\cdots&b_2^{(2)} & b_2^{(1)}&  a_3 }}\,.$$
Then by Lemma \ref{lem:disjoint from T2}, $Q_0 \not \in\widetilde T^{+,2}$, and is not of the form $R_1$, $R_2$  (because $b_1 \not \succ c_1$).
\begin{enumerate}
    \item[$\bullet$] If $Q_0 \not \in\widetilde{\mathcal T}^{+,1}$, then we let $\phi_2(S)=Q_0$.

    \item[$\bullet$] If $Q_0 \in   \widetilde{\mathcal T}^{+,1}$, then we let
  $$\phi_2(S)=  \scriptsize{ \ttableau{a_1& b_1^{(m-3)}& \cdots&b_1^{(2)}&b_1^{(1)}&{a_2}& c_1&d_1^{(k-1)}& \cdots& d_1^{(2)}& d_1^{(1)}&d_1& \red{b_1}\\
                                            \red{e_1}&b_2^{(m-3)}&\cdots&b_2^{(2)} & b_2^{(1)}&  a_3 }}=:Q_1.$$

  Then $Q_1$ is an $f$-tableau (because $a_3 \not\succ e_1$, $b_2^{(m-3)} \not \prec e_1$ and $d_1 \not \succ b_1$), that is contained neither in $\widetilde{\mathcal T}^{+,1}$ nor in $\widetilde{\mathcal T}^{+,2}$ (because $c_1 \not\prec b_1$), and is not of the form $R_0$, $R_1$, $R_2$  or $Q_0$ (because $e_1 \not\prec a_3$, $a_3 \not\succ b_1$, $c_1 \not\prec b_1$, and $e_1 \succ c_1$).

\end{enumerate}
\item [(3)] When $b_1 \not\prec a_3$ and $b_1   \succ b_2^{(m-3)}$, then  we let $\phi_2(S)=\widetilde{\sigma}_{3\rightarrow 2}^\square(S)$.

Then by Lemma \ref{lem:disjoint from T2} and Lemma \ref{lem:disjoint from T1}, $\phi_2(S)$ is contained neither  in $\widetilde{\mathcal T}^{+,1}$ nor in  $\widetilde{\mathcal T}^{+,2}$, and is not of the form $R_1$, $R_2$  or $Q_1$ (because  $\phi_2(S)$ is an $f$-tableau
$$
\scriptsize{\ttableau{a_1& b_1^{(m-3)} & \cdots &b_1^{(2)}& b_1^{(1)} &b_1& c_1& d_1^{(k-1)}&\cdots& d_1^{(2)}&d_1^{(1)}&d_1& e_1\\
                                            a_2& b_2^{(m-3)} &  \cdots &b_2^{(2)} & b_2^{(1)} & b_2}}\,. $$
satisfying $e_1 \in P_3$, $c_1 \not\prec e_1$, and  $b_2^{(m-3)} \prec e_1$).
\end{enumerate}

 We proved that $\phi_2(S)$ is not an image of $\phi_1$ for any $S$, by considering all cases that appear in the definition of $\phi_2$, Definition~\ref{def:phi2}. This completes the proof of Proposition~\ref{prop:phi2}.
\end{proof}

\bigskip


\section{Concluding Remarks}\label{sec:concluding_remarks}

We proved the Stanley-Stembridge conjecture for the natural unit interval orders corresponding to the Hessenberg functions with bounce number $3$ in the current paper. There are a few concluding remarks.

\begin{enumerate}
\item   We can give a simple proof of the Stanley-Stembridge conjecture when the bounce number $b(f)$ of a given Hessenberg function $f$ is $2$: For any partition $\mu=(\mu_1, \mu_2)$ with two parts, we have $\mathcal C_\mu=\{\mu\}$ and $\mathcal{K}_\mu=\{((\mu_1, \mu_2)^+, (\mu_1-1, \mu_2+1)^- \}$. Moreover, $\sigma_{2\rightarrow 1}^\square$ is an injection from $\mathcal{T}(\mu_1-1, \mu_2+1)^-$ to  $\mathcal{T}(\mu_1, \mu_2)^+$, as one can see in Figure~\ref{fig:bounce2}.
\begin{figure}[ht]
\begin{center}
$ \mathcal{K}_\mu=
\begin{tikzcd}[row sep=huge, column sep=tiny]
(\mu_1-1, \mu_2+1)^-\arrow[d,"\sigma_{2,1}^1"] \\
(\mu_1, \mu_2)^+
\end{tikzcd}
$\quad\qquad\quad\quad\qquad\quad
$\begin{tikzcd}[row sep=huge, column sep=tiny]
\mathcal{T}(\mu_1-1, \mu_2+1)^-\arrow[d, hook, "\sigma_{2\rightarrow 1}^\square"] \\
\mathcal{T}(\mu_1, \mu_2)^+
\end{tikzcd}
$
\end{center}
\caption{\label{fig:bounce2} Diagrams for the case when bounce number is $2$. }
\end{figure}

\item Our work done in Section~\ref{sec:bounce3} to write the coefficients in the $h$-expansion of the chromatic symmetric functions as a signed sum of the number of dual $P$-tableaux can be extended to the general cases with arbitrary bounce number.  We, however, were not able to extend the work to construct sign reversing involutions in Section~\ref{sec:h-positivity} to general cases.
\item The injections defined for the proof of the $h$-positivity are not weight(ascent) preserving. Hence our proof does not give a proof of the \emph{refined} Stanley-Stembridge conjecture: Conjecture~\ref{conj:h-positive}. We think that it would be the case that weight preserving injections could be defined in a more natural way than the ones we defined for non-refined cases.

\end{enumerate}


\begin{thebibliography}{10}

\bibitem{BC}
P.~Brosnan and T.~Y. Chow, \emph{Unit interval orders and the dot action on the
  cohomology of regular semisimple {H}essenberg varieties}, Adv. Math.
  \textbf{329} (2018), 955--1001. \MR{3783432}

\bibitem{CH}
S.~Cho and J.~ Huh, \emph{On {$e$}-positivity and {$e$}-unimodality of
  chromatic quasi-symmetric functions}, SIAM J. Discrete Math. \textbf{33}
  (2019), no.~4, 2286--2315.

\bibitem{DW}
S.~Dahlberg and S.~van Willigenburg, \emph{Lollipop and lariat symmetric
  functions}, SIAM J. Discrete Math. \textbf{32} (2018), no.~2, 1029--1039.

\bibitem{Gasha}
V.~Gasharov, \emph{Incomparability graphs of {$(3+1)$}-free posets are
  {$s$}-positive}, Discrete Math. \textbf{157} (1996), no.~1-3, 193--197.

\bibitem{GS}
D.~D. Gebhard and B.~E. Sagan, \emph{A chromatic symmetric function in
  noncommuting variables}, J. Algebraic Combin. \textbf{13} (2001), no.~3,
  227--255.

\bibitem{G-P1}
M.~Guay-Paquet, \emph{A modular law for the chromatic symmetric functions of
  $(3+1)$-free posets}, preprint, arXiv:1306.2400v1.

\bibitem{G-P2}
\bysame, \emph{A second proof of the {S}hareshian--{W}achs conjecture, by way
  of a new {H}opf algebra}, preprint, arXiv:1601.05498.

\bibitem{HP}
M. Harada and M.~E. Precup, \emph{The cohomology of abelian
  {H}essenberg varieties and the {S}tanley-{S}tembridge conjecture}, Algebr.
  Comb. \textbf{2} (2019), no.~6, 1059--1108.

\bibitem{Loe}
N.~A. Loehr, \emph{Conjectured statistics for the higher
  {$q,t$}-{C}atalan sequences}, Electron. J. Combin. \textbf{12} (2005),
  Research Paper 9, 54.

\bibitem{Mac}
I.~G. Macdonald, \emph{Symmetric functions and {H}all polynomials}, second ed.,
  Oxford Mathematical Monographs, The Clarendon Press, Oxford University Press,
  New York, 1995, With contributions by A. Zelevinsky, Oxford Science
  Publications.

\bibitem{SW}
J.~Shareshian and M.~Wachs, \emph{Chromatic quasisymmetric functions}, Adv.
  Math. \textbf{295} (2016), 497--551.

\bibitem{S1}
R.~P. Stanley, \emph{A symmetric function generalization of the chromatic
  polynomial of a graph}, Adv. Math. \textbf{111} (1995), no.~1, 166--194.

\bibitem{SS}
R.~P. Stanley and J.~R. Stembridge, \emph{On immanants of {J}acobi-{T}rudi
  matrices and permutations with restricted position}, J. Combin. Theory Ser. A
  \textbf{62} (1993), no.~2, 261--279.

\bibitem{T2}
J.~S. Tymoczko, \emph{Permutation actions on equivariant cohomology of flag
  varieties}, Toric topology, Contemp. Math., vol. 460, Amer. Math. Soc.,
  Providence, RI, 2008, pp.~365--384.

\bibitem{T1}
\bysame, \emph{Permutation representations on {S}chubert varieties}, Amer. J.
  Math. \textbf{130} (2008), no.~5, 1171--1194.

\end{thebibliography}

\providecommand{\bysame}{\leavevmode\hbox to3em{\hrulefill}\thinspace}
\providecommand{\MR}{\relax\ifhmode\unskip\space\fi MR }
\providecommand{\MRhref}[2]{%
  \href{http://www.ams.org/mathscinet-getitem?mr=#1}{#2}
}
\providecommand{\href}[2]{#2}

\end{document}